\documentclass[a4paper]{amsart}
\usepackage[T1]{fontenc}

\usepackage{amsmath,amsrefs,amssymb,ifthen,mathrsfs,subfig}
\usepackage{enumitem}
\usepackage{mathtools}
\usepackage{hyperref}
\usepackage{tikz-cd}
\usepackage{xparse}

\newtheorem{thmintro}{Theorem}

\newtheorem{propintro}[thmintro]{Proposition}

\newtheorem{questintro}[thmintro]{Question}

\theoremstyle{definition}

\newtheorem{exintro}[thmintro]{Example}

\newtheorem{thm}{Theorem}[section]

\newtheorem{prop}[thm]{Proposition}
\newtheorem{lem}[thm]{Lemma}
\newtheorem{cor}[thm]{Corollary}

\newtheorem{question}[thm]{Question}
\newtheorem{setting}[thm]{Setting}

\theoremstyle{remark}
\newtheorem{rem}[thm]{Remark}
\newtheorem{example}[thm]{Example}

\theoremstyle{definition}
\newtheorem{claim}[thm]{Claim}
\newtheorem{defi}[thm]{Definition}

\newcommand{\Z}{\mathbb{Z}}

\newcommand{\R}{\mathbb{R}}
\newcommand{\N}{\mathbb{N}}

\DeclarePairedDelimiter{\abs}{\lvert}{\rvert}
\DeclarePairedDelimiter{\Span}{\langle}{\rangle}

\newcommand{\axis}[1]{\operatorname{Axis}(#1)}
\newcommand{\actson}{\curvearrowright}

\newcommand{\cat}[1]{\ensuremath{\operatorname{CAT}(#1)}}

\newcommand{\centr}[2]{\operatorname{C}_{#1}(#2)}

\newcommand{\conj}[2]{#2 #1 {#2}^{-1}}

\newcommand{\corner}[1]{\operatorname{Cor}(#1)}
\newcommand{\cst}[1]{\operatorname{CST}(#1)}
\newcommand{\curve}[1]{\mathcal C (#1)}

\newcommand{\da}[1]{\operatorname{DA}_{#1}}

\newcommand{\edge}[1]{\operatorname{E}(#1)}

\renewcommand{\epsilon}{\varepsilon}
\newcommand{\ext}[1]{\operatorname{ext}({#1})}

\renewcommand{\phi}{\varphi}
\newcommand{\fix}[1]{\operatorname{Fix}(#1)}

\newcommand{\girth}[1]{\operatorname{girth}(#1)}

\renewcommand{\L}{\mathscr L}
\newcommand{\lab}[1]{\operatorname{label}(#1)}

\newcommand{\len}[2]{\operatorname{len}_{#1}(#2)}
\newcommand{\link}[2]{\operatorname{Link}_{#1}({#2})}

\newcommand{\norm}[2]{\operatorname{N}_{#1}(#2)}

\renewcommand{\setminus}{\smallsetminus}
\renewcommand{\sl}[1]{\operatorname{SL}(#1)}
\newcommand{\stab}[2]{\operatorname{Stab}_{#1}(#2)}

\newcommand{\ver}[1]{\operatorname{V}(#1)}

\newcommand{\wg}[1]{\operatorname{wg}(#1)}

\usepackage{xcolor}

\begin{document}

\title{Isomorphism invariance of the girth of Artin groups}

\author[G. Sartori]{Giovanni Sartori}
   \address{(Giovanni Sartori) Department of Mathematics, Heriot-Watt University and Max\-well Institute for Mathematical Sciences, Edinburgh, UK}
   \email{gs2057@hw.ac.uk}

\begin{abstract}
    For all Artin groups, we characterise the girth (i.e. the length of a shortest cycle) of the defining graph algebraically, showing that it is an isomorphism invariant. Using this result, we prove that the Artin groups based on a cycle graph are isomorphically rigid.\par
    Alongside the girth, we introduce a new graph invariant, the weighted girth, which takes into account the labels of the defining graph. Within the class of two-dimensional Artin groups of hyperbolic type, we characterise the weighted girth in terms of certain minimal right-angled Artin subgroups, showing that it is an isomorphism invariant. Finally, under the further hypothesis of leafless defining graph, we recover the weighted girth as the girth of the commutation graph introduced by Hagen-Martin-Sisto.
\end{abstract}

\maketitle

\small



\normalsize

\setcounter{tocdepth}{1}
\tableofcontents

\section*{Introduction}

\subsection*{Motivation} Artin groups are a class of groups introduced by Tits that appear in many areas of mathematics, as they generalise free groups, free abelian groups and braid groups, and are tightly related to Coxeter groups~\cite{tits1966normalisateurs}. The usual way to define an Artin group $A_\Gamma$ is to give a presentation via a finite labelled simplicial graph~$\Gamma$ or, how it is usually referred to, a \emph{defining graph}. While isomorphic defining graphs give rise to isomorphic Artin groups, the converse is false in general. In fact, one of the main open problems for this family of groups is the {isomorphism problem}:

\begin{questintro}[\cite{charney2016problems}*{Problem~28}]
    Given two labelled graphs $\Gamma$ and $\Lambda$, is there an algorithm to determine whether the Artin groups $A_\Gamma$ and $A_\Lambda$ are isomorphic?
\end{questintro}

So far, the general problem is completely open, whereas some results hold, when restricting to some subclasses~\citelist{\cite{baudisch1981subgroups}\cite{droms1987isomorphisms}\cite{vaskou2022isomorphism}}.\par
One step towards a solution of the isomorphism problem is to find isomorphism invariants of the defining graph or, in other words, properties of the defining graph that are preserved under group isomorphisms. Baudisch proved that being right-angled is an isomorphism invariant~\cite{baudisch1981subgroups}. Martin-Vaskou proved that being of large-type is preserved under isomorphism and, more precisely, that the set of label of a defining graph is also an isomorphism invariant within the class of two-dimensional Artin groups~\cite{martin2024characterising}. More recently, Jones-Mangioni-Sartori proved that having a separating vertex is an isomorphism invariant (in fact, isomorphic Artin groups have sets of big chunks in bijection)~\cite{jones2025jsj}. The isomorphism invariance of some elementary graph-theoretic properties remains open in general:

\begin{questintro}
    Let $\Gamma$ and $\Lambda$ be labelled graphs such that the groups $A_\Gamma$ and $A_\Lambda$ are isomorphic. Do $\Gamma$ and $\Lambda$ have the same number of vertices? Do they have the same number of edges? Do they have the same girth?
\end{questintro}

It is conjectured that, if two Artin groups are isomorphic, then the corresponding defining graphs are \emph{twist equivalent} (and we refer to~\cite{crisp2005automorphisms} for the definition). Being twist equivalent is an equivalence relation amongst labelled graph which is strictly weaker than isomorphism. Nevertheless, two twist-equivalent labelled graphs have, in particular, the same number of vertices and edges, and the same girth (i.e. the length of the shortest cycle subgraph). 

\subsection*{Statement of results} In this paper, we give an algebraic characterisation of the girth of the defining graph, showing that it is an isomorphism invariant.

\begin{thmintro}[Theorem~\ref{thm:girIsoInv}]
    \label{thmintro:girIsoInv}
    Let $\Gamma$ be a labelled graph; the girth of~$\Gamma$ coincides with the minimum $n\in\N_{\ge3}$ such that there exists a labelled cycle $C_n$ on $n$ vertices and an embedding $A_{C_n}\to A_\Gamma$.\par
    In particular, the girth of the defining graph is an isomorphism invariant.
\end{thmintro}

\begin{figure}
    \centering
    \includegraphics[width=0.5\linewidth]{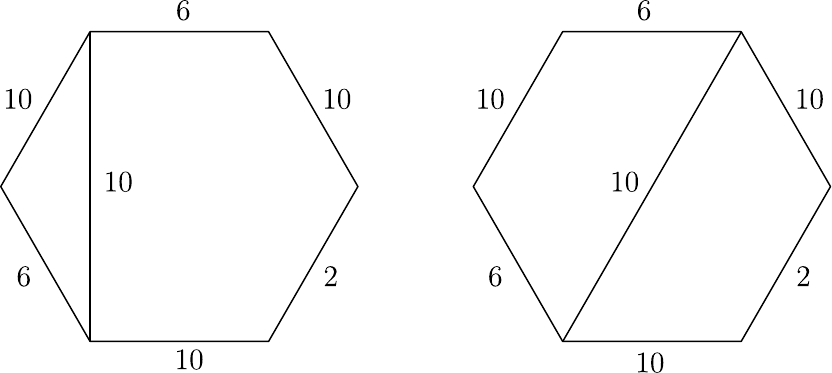}
    \caption{Two labelled graphs of different girth.}
    \label{fig:girthDistinguish}
\end{figure}

\begin{exintro}
    \label{exintro:girthDistinguish}
    The two labelled graphs of Figure~\ref{fig:girthDistinguish} have different girth ($3$ and $4$), therefore the associated Artin groups are not isomorphic, in view of Theorem~\ref{thmintro:girIsoInv}. \par 
    
    Note that the two Artin groups could not be told apart using other existing methods. Both graphs are connected and have the same number of vertices and edges. Both graphs are even-labelled, hence the abelianisations of the two groups coincide. The set of labels does not satisfy the hypotheses of Blasco-García--Paris's isomorphism criterion for even Artin groups~\cite{BLASCOGARCIA202235}*{Theorem 5.1}. None of them is of large type but both are two dimensional. Vaskou's criterion for two-dimensional Artin groups applies~\cite{vaskou2022isomorphism}*{Theorem E}, yet does not distinguish the two groups: none of them has euclidean triangles, the sets of labels with counted multiplicity coincide and all pieces (a local information at each vertex) coincide. Finally, none of the two graphs has a separating vertex~\cite{jones2025jsj}*{Theorem A}.
\end{exintro}

Our methods also allow us to show that Artin groups based on a cycle are isomorphically rigid.

\begin{thmintro}[Theorem~\ref{thm:rigidCycArtGrp}]
    \label{thmintro:rigidCycArtGrp}
    Let $\Gamma$ and $\Lambda$ be labelled cycle graphs; if $A_\Gamma\cong A_\Lambda$, then $\Gamma\cong\Lambda$.
\end{thmintro}

Alongside the girth, we consider another graph invariant, namely the \emph{weighted girth} of the defining graph, i.e. the girth of the graph which is obtained after subdividing once each edge of label at least~$3$ (see Definition~\ref{def:wg}). Throughout this paper, if $\Lambda$ is a simplicial graph, then we denote by $R_\Lambda$ the right-angled Artin group associated with~$\Lambda$.

\begin{thmintro}[Theorem~\ref{thm:wgIsoInv}]
    \label{thmintro:wgIsoInv}
    Let $A_\Gamma$ be two-dimensional Artin group of hyperbolic type. The weighted girth of~$\Gamma$ coincides with the minimum $n\in\N_{\ge3}$ such that there exists a cycle $C_n$ on $n$ vertices and an embedding $R_{C_n}\to A_\Gamma$. \par
    In particular, the weighted girth of a defining graph is an isomorphism invariant in the class of two-dimensional Artin groups of hyperbolic type.
\end{thmintro}

\begin{figure}
    \centering
    \includegraphics[width=0.5\linewidth]{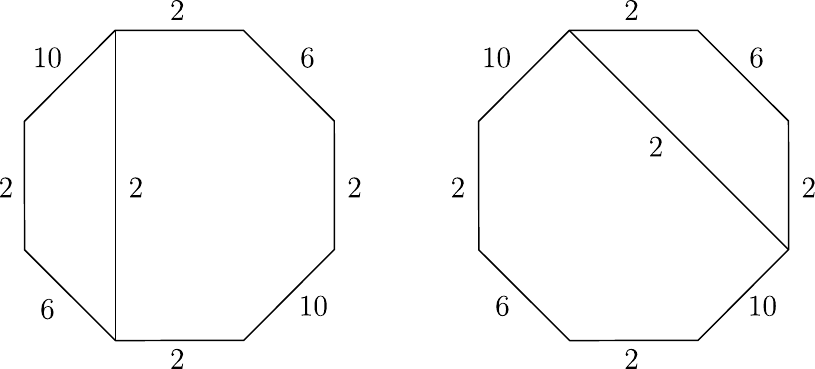}
    \caption{Two labelled graphs with same girth but different weighted girth.}
    \label{fig:wgDistinguish}
\end{figure}

\begin{exintro}
    The two graphs in Figure~\ref{fig:wgDistinguish} have different weighted girth ($6$ and $5$). Both graphs are triangle free and do not contain the $(2,2,2,2)$-square subgraph, hence are two dimensional and of hyperbolic type (see Definition~\ref{def:hypTyp}). Theorem~\ref{thmintro:wgIsoInv} applies, telling apart the associated Artin groups. \par
    Note that the two graphs have the same girth, so Theorem~\ref{thmintro:girIsoInv} (as well as all other criteria listed in Example~\ref{exintro:girthDistinguish}) fails to distinguish the two groups. 
\end{exintro}

Finally, by means of Theorem~\ref{thmintro:wgIsoInv}, we realise the weighted girth $\wg\Gamma$ as the girth of the commutation graph~$Y_\Gamma$. The latter was introduced by Hagen-Martin-Sisto to show that Artin groups of large and hyperbolic type are hierarchically hyperbolic~\cite{hagen2024extra}, playing the role of the curve graph for mapping class groups, and is quasi-isometric to the cone-off Deligne complex of Martin-Przytycki~\cite{martin2022acylindrical}. Moreover, when $\Gamma$ is furthermore without leaves, the commutation graph coincides with the intersection graph introduced by Huang-Osajda-Vaskou~\cite{huang2024rigidity}.

\begin{thmintro}[Theorem~\ref{thm:girthCrvGrph}]
    \label{thmintro:girthCrvGrph}
    Let $A_\Gamma$ be a two-dimensional Artin group of hyperbolic type and assume that $\Gamma$ has no leaves. Then $\wg\Gamma=\girth{Y_\Gamma}$.
\end{thmintro}

\subsection*{Strategy of proof} Our methods are mainly geometric. The proofs of Theorem~\ref{thmintro:girIsoInv} and Theorem~\ref{thmintro:wgIsoInv} are based on the observation that, when $A_\Gamma$ is two dimensional and of hyperbolic type, the embedding of a cycle Artin group in $A_\Gamma$ induces a \emph{cycle of standard trees} in the Deligne complex~$D_\Gamma$ (see Lemma~\ref{lem:mono->cycle}). A standard tree is, up to translation, the fixed-point set of a standard generator and by cycle of standard trees we mean a sequence $\{T_i\}_{i\in\Z_n}$ of standard trees with the properties that, for every $i\in\Z_n$, $T_i\cap T_{i+1}\ne\varnothing$ and $T_{i-1}\cap T_i\cap T_{i+1}=\varnothing$ (see Definition~\ref{def:cycStdTree}). Remarkably, our notion of cycle of standard trees is weaker than the one introduced by Blufstein-Martin-Vaskou~\cite{blufstein2024homomorphisms} (see Remark~\ref{rem:cycStdTrsOfBMV}). We then define a meaningful version of length and weighted girth of a cycle of standard trees (see Definition~\ref{def:wgCycStdTrs}), and show that these have uniform lower bounds that depend only on~$\Gamma$. This will be our main technical result:

\begin{propintro}[\protect{Lemma~\ref{lem:lenCycStdTrs} and Lemma~\ref{lem:wgCycStdTrs}}]
    \label{propintro:stdTrsVsDefGrph}
    Let $A_\Gamma$ be a two-dimensional Artin group. For every cycle of standard trees $\L=\{T_i\}_{i\in\Z_n}$ in $ D_\Gamma$,
    \[
    n\ge\girth\Gamma.
    \]
    Moreover, if $\Gamma$ is triangle free or of hyperbolic type, then
    \[
    \wg\L\ge\wg\Gamma.
    \]
\end{propintro}

\subsection*{Further questions} We end this introduction with some questions that arise from our work.

\subsubsection*{Complete invariance of weighted girth} It is in general not know whether the weight\-ed girth of the defining graph is an isomorphism invariant within the class of all Artin groups:

\begin{questintro}
    \label{quest:isWgInv}
    Let $\Gamma$ and $\Lambda$ be labelled graphs. If the Artin groups $A_\Gamma$ and $A_\Lambda$ are isomorphic, is it true that $\wg\Gamma=\wg\Lambda$?
\end{questintro}

Note that, in view of Theorem~\ref{thmintro:girIsoInv}, $\girth\Gamma=\girth\Lambda$ and, if such girth is at least~$5$, then $A_\Gamma$ and $A_\Lambda$ are both of hyperbolic type, in which case Theorem~\ref{thmintro:wgIsoInv} applies, giving $\wg\Gamma=\wg\Lambda$. Therefore Question~\ref{quest:isWgInv} reduces to understanding the ``small girth'' cases. When $A_\Gamma$ has girth~$3$ or~$4$, it is not necessarily two dimensional of hyperbolic type and the cycle-of-standard-trees arguments described above do not extend. Remarkably, we can circumvent this limitation in proving Theorem~\ref{thmintro:girIsoInv}: the case case of girth~$3$ is minimal by definition (cycles consist of at least three vertices), while an ad-hoc argument applies to triangle-free graphs that contain a $(2,2,2,2)$-square, which are two dimensional but not of hyperbolic type.

\subsubsection*{Recognising cycle Artin subgroups} As hinted in the ``Strategy of proof'' paragraph, if $C$ is a labelled cycle and $A_\Gamma$ is a two-dimensional Artin group of hyperbolic type, then any group monomorphism $A_C\to A_\Gamma$ induces a cycle of standard trees in~$D_\Gamma$. It is natural to ask the converse question:

\begin{questintro}
    Let $A_\Gamma$ be a two-dimensional Artin group of hyperbolic type and let $\L\subseteq D_\Gamma$ be a (simple) cycle of standard trees. Does there exist a labelled cycle $C$ and a group monomorphism $\phi\colon A_C\to A_\Gamma$ such that $\L$ is the cycle of standard trees associated with~$\phi$?
\end{questintro}

We refer to Question~\ref{quest:classifyCycArtSubGrp} and Question~\ref{quest:classifyCycRaagSubGrp} for two more refined versions of the question.\medskip

\subsubsection*{Spectrum of cycle Artin subgroups} To get a finer invariant than girth, one may consider the whole spectrum $\operatorname{spec}(\Gamma)$,
\[
\{n\in\N:\text{there exists a labelled $n$-cycle $C_n$ and an embedding $A_{C_n}\to A_\Gamma$}\}.
\]
Besides realising $\min\operatorname{spec}(\Gamma)$ as $\girth\Gamma$ in Theorem~\ref{thmintro:girIsoInv}, we are not aware of works directed to the study of this object.

\subsection*{Organisation of the paper} In Section~\ref{sec:genBackGround} we recall some basic definitions and results on Artin groups, specialising on the family of those of dimension~$2$. \par
As a first step to understand what cycle Artin groups embed in two-dimension\-al Artin groups of hyperbolic type, we first need to classify the RAAG subgroups of dihedral Artin groups. We do this in Section~\ref{sec:raagSubGrpOfDih}.\par
In Section~\ref{sec:embeddCycArtGrp} we introduce the definition of a cycle of standard trees and show that embeddings of cycle Artin groups determine cycles of standard trees in the Deligne complex.\par
In Section~\ref{sec:comb gbt} we show how to associate a disc diagram to a cycle of standard trees and recall the combinatorial version of the Gauss--Bonnet theorem, which is our main technical tool in the paper. In Section~\ref{sec:lenAndWgCycStdTrs} we apply the combinatorial Gauss--Bonnet theorem to such disc diagram to obtain Proposition~\ref{propintro:stdTrsVsDefGrph}.\par
Finally, we derive applications of Proposition~\ref{propintro:stdTrsVsDefGrph}: in Section~\ref{sec:minCycEmbedd} we give the algebraic characterisation of girth (Theorem~\ref{thmintro:girIsoInv}) and weighted girth (Theorem~\ref{thmintro:wgIsoInv}); in Section~\ref{sec:girthCrvGrph} we realise the weighted girth of the defining graph as the girth  of the commutation graph (Theorem~\ref{thmintro:girthCrvGrph}); in Section~\ref{sec:rigiCycArtGrp} we show the isomorphism rigidity of Artin groups whose defining graph is a cycle (Theorem~\ref{thmintro:rigidCycArtGrp}). 

\subsection*{Acknowledgements} I would like to thank Alexandre Martin for his support, and Jacques Darné, Federica Gavazzi, Oli Jones, Giorgio Mangioni and Nicolas Vaskou for interesting conversations. This work was partially supported by the EPSRC Standard Research Grant UKRI1018.

\section{Background on Artin groups}
\label{sec:genBackGround}

We recall here some well-known facts about Artin groups.

\subsection{Artin groups}\label{subsec:artinBG} Throughout the paper we assume that all the defining graphs of our Artin groups are connected, unless otherwise specified. This is no loss of generality, for having a disconnected defining graph is equivalent to splitting as a free product~\cite{behrstock2009thick}*{Proposition 1.3}, and therefore is an isomorphism invariant.  

\begin{defi}[Artin groups]
    Let $\Gamma=(\ver\Gamma,\edge\Gamma)$ be a finite simplicial graph;
    \begin{enumerate}
        \item A \emph{labelling} of~$\Gamma$ is a map $m\colon\edge\Gamma\to\N_{\ge2}$. If $\{s,t\}$ is an edge of~$\Gamma$, then we write $m_{st}$ in place of $m(\{s,t\})$. We call the pair $(\Gamma,m)$ a defining graph\footnote{The terms \emph{Coxeter graph} and \emph{presentation graph} are also common in the literature.};
        \item Let $(\Gamma,m)$ be a defining graph; the \emph{Artin group} associated with $(\Gamma,m)$ is the group $A_\Gamma$ with presentation
        \[
        \Span{\ver\Gamma\mid\text{for every $\{s,t\}\in\edge\Gamma$, }\underbrace{sts\cdots}_{m_{st}}=\underbrace{tst\cdots}_{m_{st}}}.
        \]
        The elements of $\ver\Gamma$ are often called the \emph{standard generators} of the presentation;
        \item the \emph{Coxeter group} associated with $\Gamma$ is the quotient $W_\Gamma$ of $A_\Gamma$ by the normal closure of the subgroup generated by $\{s^2:s\in\ver\Gamma\}$.
    \end{enumerate}
\end{defi}

We will often forget about the labelling and refer to $A_\Gamma$ as the Artin group associated with~$\Gamma$. \par 

If~$\Gamma$ is a graph and $\Lambda\subseteq\Gamma$ is a subgraph, recall that~$\Lambda$ is said to be the \emph{full subgraph of~$\Gamma$ induced by $\ver\Lambda$} (or just a full subgraph, or an induced subgraph) if every edge of~$\Gamma$ connecting two vertices of~$\Lambda$ is also an edge of~$\Lambda$.\par
For every subset $S$ of vertices of~$\Gamma$ we can consider the subgroup~$\Span{S}_{A_\Gamma}$ of $A_\Gamma$ generated by $S$. Such a subgroup is called the \emph{standard parabolic subgroup} of $A_\Gamma$ generated by~$S$ and is denoted $A_S$. If~$\Gamma_S$ denotes the full subgraph of~$\Gamma$ spanned by~$S$, then $A_S$ is  isomorphic to~$A_{\Gamma_S}$ by a result of van der Lek~\cite{van1983homotopy}. A subgroup of $A_\Gamma$ that is conjugated to a standard parabolic subgroup is called a \emph{parabolic subgroup} of $A_\Gamma$. 

\begin{defi}[families of Artin groups]
    An Artin group $A_\Gamma$ is said to be:
    \begin{enumerate}
        \item \emph{dihedral} if $\abs{\ver\Gamma}=2$;
        \item \emph{right angled} if all the labels of $\Gamma$ are $2$'s, in which case we say that $A_\Gamma$ is a right-angled Artin group (RAAG);
        \item of \emph{spherical type} if the associated Coxeter group $W_\Gamma$ is finite;
        \item \emph{two dimensional} if every standard parabolic subgroup of $A_\Gamma$ of spherical type corresponds to a subgraph of at most two vertices.
    \end{enumerate}
\end{defi}

Note that these properties are, in principle, properties of the defining graph, and not group theoretical. In fact, being right angled and two dimensional are actually group-theoretical properties.\par
Baudisch proved that in a right-angled Artin group every two elements either commute or generate a non-abelian free subgroup~\cite{baudisch1981subgroups}. It follows that, if $A_\Gamma$ and $A_\Lambda$ are isomorphic Artin groups and $A_\Gamma$ is a RAAG, then so is $A_\Lambda$. In view of this observation, we will write $R_\Gamma$ in place of $A_\Gamma$, whenever the latter is a RAAG. (Such choice of notation is slightly unconventional but helpful in the arguments of this paper.)\par
Being two dimensional is a group-theoretical property by the following characterisation, which is known to the experts. We streamline the proof for the sake of completeness.

\begin{prop}
    \label{prop:2d_iso_inv}
    The following conditions are equivalent for an Artin group $A_\Gamma$:
    \begin{enumerate}
        \item $A_\Gamma$ is two-dimensional;
        \item if there is $H\le A_\Gamma$ such that $H\cong\Z^n$, then $n\le 2$.
    \end{enumerate}
    In particular, the property of being two-dimensional is an isomorphism invariant in the class of Artin groups.
\end{prop}

\begin{proof}
    Direct implication. By the solution of the $\operatorname{K}(\pi,1)$-conjecture for two-dimen-sional Artin groups~\cite{charney1995k}, $A_\Gamma$ admits as a classifying space a CW-complex of dimension at most~$2$, i.e. has geometric dimension at most~$2$. The latter gives an upper bound on the cohomological dimension of $A_\Gamma$ (and on the maximal dimension of a free abelian subgroup, in particular).\par
    Reverse implication. Assume that $A_\Gamma$ is not two-dimensional. Then there exists a triangle subgraph $\Delta\subseteq\Gamma$ such that the parabolic subgroup $A_\Delta$ is of spherical type. It is not a loss in generality to assume that at most one of the edges of $\Delta$ has label~$2$, for otherwise $A_\Delta$ would split as a direct product of standard parabolic subgroups (i.e. $A_\Delta$ would be \emph{reducible}). Let $e$ be and edge of $\Delta$ of label at least~$3$ and let $v$ be a vertex of $e$. The groups $A_v$, $A_e$ and $A_\Delta$ are irreducible and of spherical type and hence their centres are infinite cyclic~\cite{deligne1972immeubles}*{Theorem 4.21}, say generated by $v$, $z_e$ and $z_\Delta$ respectively. Then the subgroup of $A_\Gamma$ generated by $\{v,z_e,z_\Delta\}$ is isomorphic to~$\Z^3$.
\end{proof}

\subsection{Deligne complex} We use the parabolic subgroups of $A_\Gamma$ to construct a simplicial complex $A_\Gamma$ acts on. 

\begin{defi}[Deligne complex]
    Let $\Gamma$ be a finite labelled graph; the \emph{simplicial Deligne complex} (or simply the Deligne complex) $D_\Gamma$ associated with $\Gamma$ is the simplicial realisation of the partially ordered set
    \[
    \{g\cdot A_\Lambda:\text{$g\in A_\Gamma$ and $\Lambda\subseteq\Gamma$ such that $A_{\Lambda}$ is of spherical type}\},
    \]
    ordered by inclusion.
\end{defi}

More explicitly, $D_\Gamma$ is the simplicial complex defined as follows:
\begin{itemize}
    \item vertices of $D_\Gamma$ are in bijection with left cosets of standard parabolic subgroups of $A_\Gamma$ of spherical type;
    \item for every $n\in\N$ and every $g\in A_\Gamma$, the vertices $g\cdot A_{\Gamma_0},\dots,g\cdot A_{\Gamma_n}$ span an $n$-simplex of $D_\Gamma$ if $\Gamma_0\subsetneq\dots\subsetneq\Gamma_n$. 
\end{itemize}

The group $A_\Gamma$ acts on $D_\Gamma$ by left multiplication and the action is cobounded and without inversion, i.e. if an element of $A_\Gamma$ fixes a simplex set-wise, then it fixes it point-wise. We denote by $K_\Gamma$ the subcomplex of $D_\Gamma$ that corresponds to the simplicial realisation of 
\[
\{A_\Lambda:\text{$\Lambda\subseteq\Gamma$ such that $A_\Lambda$ is of spherical type}\}
\]
and call it the \emph{canonical fundamental domain} for the action $A_\Gamma\actson D_\Gamma$. It is straightforward to check that $K_\Gamma$ is indeed a fundamental domain.\par
For every induced subgraph $\gamma$ of $\Gamma$, its simplicial realisation $K_\gamma$ is defined analogously and it is a subcomplex of $K_\Gamma$. \par

When $A_\Gamma$ is a two-dimensional Artin group, the parabolic subgroups of $A_\Gamma$ of sperical type are based on at most two vertices, whence follows that $D_\Gamma$ is a simplicial complex of dimension at most two.

\begin{defi}[type of vertices]
    \label{def:type of vertices}
    Let $A_\Gamma$ be a two-dimensional Artin group;
    \begin{itemize}
        \item left cosets of the trivial subgroup are called \emph{type-zero vertices} (or vertices of type~$0$) of $D_\Gamma$;
        \item left cosets of cyclic standard parabolic subgroups are called \emph{type-one vertices} (or vertices of type~$1$) of $D_\Gamma$;
        \item left cosets of dihedral spherical-type parabolic subgroups are called \emph{type-two vertices} (or vertices of type~$2$) of $D_\Gamma$;
    \end{itemize}
\end{defi}

In particular, type-two vertices correspond with left cosets of the form $g\cdot \Span {a,b}$ for some $g\in A_\Gamma$ and some $\{a,b\}\in\edge\Gamma$ and thus we can talk of the \emph{label} $m_{ab}$ of a type-two vertex.\par 
It is often convenient to endow $D_\Gamma$ with a metric. In this paper, the simplicial Deligne complex will be given the Moussong metric, which we now introduce. For every edge $\{a,b\}\in\edge\Gamma$, we put a metric on the triangle with vertices $\Span\varnothing$, $\Span a$ and $\Span{a,b}$ as follows:
\begin{itemize}
    \item the edge $\{\Span\varnothing,\Span a\}$ is given unit length;
    \item the angle at $\Span a$ is right-angled;
    \item the angle at $\Span{a,b}$ is $\frac{\pi}{2m_{ab}}$
\end{itemize}
and the remaining edges and angles are such that such triangle is euclidean. When gluing two triangles along a common edge, the gluing map is an isometry. We refer to Figure~\ref{fig:deligneTriangle}.

\begin{figure}
    \centering
    \includegraphics[width=0.5\linewidth]{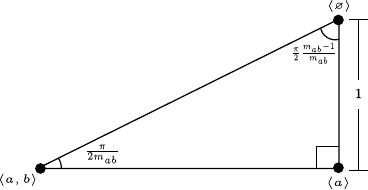}
    \caption{The Moussong metric on a triangle of $D_\Gamma$.}
    \label{fig:deligneTriangle}
\end{figure}

\begin{thm}[\citelist{\cite{charney1995k}*{Proposition~4.4.5}\cite{martin2022acylindrical}*{Section~3.1}}]
    Let $A_\Gamma$ be a two-dimension\-al Artin group. The Moussong metric on $D_\Gamma$ is a \cat 0 metric. Moreover, if $A_\Gamma$ is of hyperbolic type (see Definition~\ref{def:hypTyp}), then $D_\Gamma$ admits a \cat{-1} metric.
\end{thm}

Alongside the simplicial Deligne complex, we consider its cubical counterpart.

\begin{defi}[cubical Deligne complex]
    Let $\Gamma$ be a finite labelled graph; the \emph{cubical Deligne complex} $C_\Gamma$ associated with $\Gamma$ is the cube complex with the same vertex set as $D_\Gamma$ and where every spherical-type $k$-clique of $\Gamma$ spans a unitary $k$-cube.
\end{defi}

More explicitly, let $g\in A_\Gamma$ and let $\Lambda\subseteq\Gamma$ be a subgraph of spherical type. For every subgraph $M\subseteq\Lambda$, if $\abs{\ver\Lambda\setminus\ver M}=k$, then the vertices $\{gA_{M'}:M\subseteq M'\subseteq\Lambda\}$ span a unitary $k$-cube of $C_\Gamma$.\par
Recall that a graph is \emph{triangle free} if each of its cycle subgraphs has at least four edges.

\begin{thm}[\cite{charney1995k}*{Theorem 4.3.5}]
    Let $A_\Gamma$ be a triangle-free Artin group. The cubical Deligne complex (endowed with the standard cubical metric) is a \cat 0 cube complex.
\end{thm}

It is straightforward to see that $A_\Gamma$ acts on $D_\Gamma$ and on $C_\Gamma$ by left multiplication and the action is without inversion. The simplicial Deligne complex can be obtained by suitably subdividing the cubical Deligne complex, therefore all the results about the action of $A_\Gamma$ on $D_\Gamma$ carry over to the action on $C_\Gamma$. In particular, in dimension two the elliptic elements are fully classified.

\begin{lem}[\citelist{\cite{martin2022acylindrical}*{Section~4}\cite{hagen2024extra}*{Lemma~2.27}}]
    Let $A_\Gamma$ be a two-dimensional Artin group and let $g\in A_\Gamma\smallsetminus\{1\}$ be an element acting elliptically on $D_\Gamma$:
    \begin{enumerate} 
        \item if $g$ is conjugated to a non-trivial power of a standard generator, then its fixed-point set is a tree contained in $D_\Gamma^{(1)}$, called the \emph{standard tree} $T_g$ associated with $g$;
        \item otherwise, the fixed-point set of $g$ is precisely a vertex of type~$2$.
    \end{enumerate}
    For every standard generator $a\in\ver\Gamma$, the global stabiliser of $T_a$ coincides with the centraliser (equivalently, normaliser) of $a$ in $A_\Gamma$. Such centraliser is of the form $\Span a\times F$, where $F$ is a finite-rank free group.
\end{lem}

Based on their fixed-point sets, elliptic elements are divided into \emph{tree-elliptic} and \emph{vertex-elliptic}. We use this characterisation of elements acting elliptically on the Deligne complex to show that dihedral parabolic subgroups are inclusion-wise maximal amongst the subgroups of $A_\Gamma$ that are isomorphic to a dihedral Artin group.

\begin{lem}
    \label{lem:dih_par_are_max}
    Let $A_\Gamma$ be a two-dimensional Artin group and let $\{a,b\}$ be an edge of $\Gamma$ with label~$m_{ab}\ge3$. Suppose that there are $g\in A_\Gamma$, $m\in\N_{\ge3}$ and a subgroup $H\le A_\Gamma$ such that $\conj {\Span{a,b}} g\le H$ and $H\cong\da m$. Then $H=\conj{\Span{a,b}}g$.
\end{lem}

\begin{proof}
    Up to a conjugation in $A_\Gamma$ it is not restrictive to assume that $g=1$. Because $m,m_{ab}\ge3$, both $H$ and $\Span{a,b}$ are spherical irreducible and hence have infinite cyclic centre, say $\operatorname{Z}(H)=\Span{z_H}$ and $\operatorname{Z}(\Span{a,b})=\Span{z_{ab}}$. Since $\Span{a,b}\le H$, it follows that $z_H=z_{ab}^i$ for some $i\in\Z\setminus\{0\}$. It follows that, for every $h\in H$, in $D_\Gamma$
    \[
    h\cdot\Span{a,b}=h\cdot\fix{z_{ab}^i}=h\cdot \fix{z_H}=\fix{\conj{z_H}h}=\fix{z_H}=\fix{z_{ab}^i}=\Span{a,b},
    \]
    that is, $H$ stabilises the vertex $\Span{a,b}$ and hence $H\le\Span{a,b}$.
\end{proof}

\subsection{Hyperbolic-type Artin groups}\label{subsec:hyp-type} By \emph{hyperbolic group} we mean word hyperbolic (or Gromov-hyperbolic). Although very few Artin groups are actually hyperbolic, we can ask when the associated Coxeter group is.

\begin{defi}[hyperbolic type]
    \label{def:hypTyp}
    An Artin group $A_\Gamma$ is of \emph{hyperbolic type} if the associated Coxeter group $W_\Gamma$ is hyperbolic. 
\end{defi}

For a two-dimensional Artin group, being of hyperbolic type is equivalent to simultaneously satisfying the following two conditions~\cite{moussong1988hyperbolic}:
\begin{enumerate}
    \item every $(p,q,r)$-triangle subgraph of $\Gamma$ satisfies $\frac1p+\frac1q+\frac1r<1$;
    \item $\Gamma$ does not contain any $(2,2,2,2)$-square as a full subgraph.
\end{enumerate}

When $A_\Gamma$ is two-dimensional of hyperbolic type, the subgroups of $A_\Gamma$ that are abstractly isomorphic to a dihedral Artin group are classified. The classification of the abelian ones is due to Martin-Przytycki.

\begin{thm}
    [\cite{martin2022acylindrical}*{Corollary~C}]
    \label{thm:classify Z^2 subgrp}
    Let $A_\Gamma$ be a two-dimensional Artin group of hyperbolic type and let $H$ be a subgroup of $A_\Gamma$ that is virtually $\Z^2$. Then one of the following holds:
    \begin{itemize}
        \item $H$ is contained in a dihedral parabolic subgroup of $A_\Gamma$;
        \item $H$ is contained in the centraliser of a conjugate of a standard generator of~$A_\Gamma$. In particular, $H$ contains a non-zero power of a conjugate of a standard generator.
    \end{itemize}
\end{thm}

Turning to the non-abelian case, Vaskou gives a full classification of the subgroups of two-dimensional Artin groups that are isomorphic to dihedral Artin groups~\cite{vaskou2022isomorphism}*{Theorem~D}. From its classification, it turns out that the exotic dihedral subgroups (i.e. those that are not conjugated into a dihedral parabolic subgroup) only occur in Artin groups that are not of hyperbolic type. On the contrary, being of hyperbolic type gives a more restrictive classification. For convenience, we add a proof of this fact. 

\begin{lem}
    \label{lem:dihedral of hyp-type are parab}
    Let $A_\Gamma$ be a two-dimensional Artin group of hyperbolic type and let $H\le A_\Gamma$ be a subgroup that is virtually isomorphic to a dihedral Artin group $\da m$, for some $m\in\N_{\ge3}$. Then $H$ is contained in a dihedral parabolic subgroup.
\end{lem}

\begin{proof}
    It is not restrictive to assume that $\Gamma$ does not consist of a single edge, for otherwise the statement is trivially satisfied. Following the construction of Martin-Przytycki, let $\hat D_\Gamma$ denote the cone-off Deligne complex of $A_\Gamma$. Since $A_\Gamma$ is of hyperbolic type, $\hat D_\Gamma$ admits a \cat{-1} metric and an acylindrical action of $A_\Gamma$~\cite{martin2022acylindrical}*{Theorem A}. Because $H$ virtually splits as a direct product $\Z\times F$ of infinite groups, $H$ is not acylindrically hyperbolic~\cite{Osin:acyl}*{Corollary 7.3}. It follows that $H$ has bounded orbits and hence acts elliptically on $\hat D_\Gamma$~\citelist{\cite{Osin:acyl}*{Theorem 1.1}\cite{bridson2013metric}*{Theorem II.2.8}}. We now shall distinguish two cases, based on the vertex of $\hat D_\Gamma$ that $H$ fixes.\par
    If $H$ fixes a vertex of $D_\Gamma$, then such vertex is necessarily of type~$2$ and thus $H$ is contained in a dihedral parabolic subgroup of $A_\Gamma$.\par
    If $H$ fixes a cone point, then $H$ globally stabilises a standard tree of $D_\Gamma$, that is, there exists $a\in\ver\Gamma$ such that $H$ is conjugated inside $\centr{A_\Gamma}{a}\cong\Z\times F$. This implies that $H$ is virtually~$\Z^2$, for no generalised braid relations occur in $\Z\times F$, against the hypothesis of $m\ge3$.
\end{proof}

As a direct consequence of Lemma~\ref{lem:dih_par_are_max} and Lemma~\ref{lem:dihedral of hyp-type are parab}, we obtain the following inclusion of the sets of labels.

\begin{lem}
    \label{lem:labels_subset}
        Let $A_\Gamma$ and $A_\Lambda$ be isomorphic two-dimensional Artin groups. If $A_\Lambda$ is of hyperbolic type, then
        \[
        \{m_e:\text{$e\in\edge\Gamma$ and $m_e\ge3$}\} \subseteq \{m_e:\text{$e\in\edge\Lambda$ and $m_e\ge3$}\}.
        \]
\end{lem}

\begin{proof}
    Let $\phi\colon A_\Gamma\to A_\Lambda$ be an isomorphism and let $\{a,b\}\in\edge\Gamma$ be an edge with label $m_{ab}\ge3$. By Lemma~\ref{lem:dihedral of hyp-type are parab} there exist $g\in A_\Gamma$ and $\{x,y\}\in\edge\Lambda$ with $m_{xy}\ge3$ such that $\phi(\Span{a,b})\subseteq\conj{\Span{x,y}}{\phi(g)}$. It follows that $g^{-1}\Span{a,b}g\subseteq\phi^{-1}(\Span{x,y})$. By Lemma~\ref{lem:dih_par_are_max} $g^{-1}\Span{a,b}g=\phi^{-1}(\Span{x,y})$ and hence $m_{ab}=m_{xy}$.
\end{proof}

\section{RAAG subgroups of dihedral Artin groups}\label{sec:raagSubGrpOfDih} We characterise the RAAGs that can appear as subgroups of dihedral Artin groups. Understanding such possible subgroups of dihedral Artin groups is key to understanding the possible RAAG subgroups of two-dimensional Artin groups in general, as a local-to-global strategy.

\begin{defi}[paths, cycles, pods]
    Let $n\in\N_{\ge1}$:
    \begin{enumerate}
        \item the \emph{simplicial path} on $n$ vertices is the graph $P_n$ with set of vertices
        \[
        \ver{P_n}=\{u_i:i\in\{1,\dots,n\}\}
        \]
        and set of edges
        \[
        \edge{P_n}=\left\{\{u_i,u_{i+1}\}:i\in\{1,\dots,n-1\}\right\};
        \]
        \item for $n\ge3$, the \emph{simplicial cycle} on $n$ vertices is the graph $C_n$ with set of vertices
        \[
        \ver{C_n}=\{u_i:i\in\Z_n\}
        \]
        and set of edges
        \[
        \edge{C_n}=\left\{\{u_i,u_{i+1}\}:i\in\Z_n\right\};
        \]
        \item the \emph{simplicial $n$-pod} is the graph $S_n$ with set of vertices
        \[
        \ver{S_n}=\{u_i:i\in\{0,\dots,n\}\}
        \]
        and set of edges 
        \[
        \edge{S_n}=\{\{u_0,u_i\}:i\in\{1,\dots,n\}\}.
        \]
    \end{enumerate}
\end{defi}

For $m\in\N_{\ge2}$, let $\da m$ denote the dihedral Artin group of label~$m$. The goal of this section is to prove the following characterisation of those RAAGs on some paths $P_n$ or on some cycles $C_n$ that can appear as subgroups of $\da m$:

\begin{lem}
    \label{lem:raagSubGrpOfDihedral}
    Let $m\in\N_{\ge2}$;
    \begin{enumerate}
        \item the dihedral Artin group $\da m$ does not contain any cycle RAAGs as subgroups;
        \item if $m\ge3$ and $\da m$ contains a subgroups isomorphic to $R_{P_n}$, then $n\le3$;
        \item if $\da 2$ contains a subgroup isomorphic to $R_{P_n}$, then $n\le2$.
    \end{enumerate}
\end{lem}

Let us remark that $\da 2\cong\Z^2$ and, for $m\ge3$, $\da m$ contains a finite-index subgroup of the form $\Z\times F_k$ for some $k\in\N_{\ge2}$~\cite{crisp2005automorphisms}*{Section~2}. In particular, $\da m$ is virtually isomorphic to a RAAG $R_{S_k}$. In order to prove Lemma~\ref{lem:raagSubGrpOfDihedral}, we will first understand which RAAGs can embed in $R_{S_k}$ and then show that every RAAG subgroup of a dihedral Artin group embeds in the $R_{S_k}$ subgroup.

\begin{thm}[\cite{kim2013embedability}*{Theorem~1.11}]
    \label{thm:raag subgroups of raags}
    Let $\Gamma,\Lambda$ be finite simplicial graphs and let $\Gamma$ be triangle free; the following are equivalent:
    \begin{enumerate}
    \item there is a group monomorphism $R_\Lambda\to R_\Gamma$;
    \item $\Lambda$ is an induced subgraph of $\ext \Gamma$.
\end{enumerate}
\end{thm}

Here $\ext\Gamma$ denotes the \emph{extension graph} of $\Gamma$: its vertex set is $\{\conj vg:v\in\ver\Gamma, g\in R_\Gamma\}$ and two vertices are adjacent when the two corresponding group elements commute in $R_\Gamma$. 

\begin{lem}
    \label{lem:raag subgroups of star raag}
    Let $k\in\N_{\ge2}$:
    \begin{enumerate}
        \item the group $R_{S_k}$ does not contain any subgroup isomorphic to $R_{C_n}$;
        \item the group $R_{S_k}$ contains a subgroup isomorphic to $R_{P_n}$ if and only if $n\le3$.
    \end{enumerate}
\end{lem}

\begin{proof}
    Since $S_k$ is triangle free, by Theorem~\ref{thm:raag subgroups of raags}, the right-angled Artin subgroups of $R_{S_k}$ correspond to the induced subgraphs of $\ext{S_k}$. The pod $S_k$ decomposes as a join $\{u_0\}*\{u_1,\dots,u_k\}$ and thus $\ext{S_k}=\ext{\{u_0\}}*\ext{\{u_1,\dots,u_k\}}$~\cite{kim2013embedability}*{Lemma~3.5}. By the definition of extension graph, we have that $\ext{\{u_0\}}=\{u_0\}$, whereas $\ext{\{u_1,\dots,u_k\}}$ is a discrete graph on countably-many vertices. Thus $\ext{S_k}$ is a join of the form $\{v\}*\{v_i:i\in\Z\}$. As induced subgraphs, it does not contain any cycle and it only contains paths on at most three vertices.
\end{proof}

\begin{lem}
    \label{lem:subgroups of virtual raag}
    Let $\Gamma$ be finite simplicial labelled graph, let $\Lambda$ be a finite simplicial graph and let $G$ be a group that is virtually isomorphic to $A_\Gamma$. The following are equivalent:
    \begin{enumerate}
        \item $A_\Gamma$ contains a subgroup isomorphic to $R_\Lambda$;
        \item $G$ contains a subgroup isomorphic to $R_\Lambda$.
    \end{enumerate}
\end{lem}

\begin{proof}
    Let $H$ be a finite-index subgroup of $G$ isomorphic to $A_\Gamma$ and let $d=\abs{G:H}$. The subgroup of $G$ generated by $\{a^d:a\in\ver\Lambda\}$ is contained in $H$ and it is isomorphic to $R_\Lambda$~\cite{crisp2001solution}*{Theorem~1}. The other implication is true in general.
\end{proof}

When $G$ is a dihedral Artin group, we obtain Lemma~\ref{lem:raagSubGrpOfDihedral}.

\begin{proof}[Proof of Lemma~\ref{lem:raagSubGrpOfDihedral}]
    For $m\ge3$, $\da m$ is virtually $\Z\times F_k$ for some $k\in\N_{\ge2}$ and the result follows as a combination of Lemma~\ref{lem:raag subgroups of star raag} and Lemma~\ref{lem:subgroups of virtual raag}. For $m=2$, it is sufficient to observe that $\da 2$ coincides with $R_{P_2}$ and $\ext {P_2}=P_2$.
\end{proof}

\begin{cor}
    \label{cor:cycle subgrp dihedral}
     Let $n\in\N_{\ge3}$ and let $C_n$ be a labelled cycle. A dihedral Artin group does not contain any subgroup isomorphic to $A_{C_n}$.
\end{cor}

\begin{proof}
    Let us assume by contradiction that there exist $m\in\N_{\ge2}$ and a labelled cycle $C_n$ such that $\da m$ contains a subgroup isomorphic to $A_{C_n}$. For $i\in\Z_n$, if $m_{i,i+1}\ge3$, then we denote by $z_{i,i+1}$ the generator of the centre of $\Span{u_i,u_{i+1}}$. The subgroup of $A_{C_n}$ generated by
    \[
    \{u_i^4:i\in\Z_n\}\cup\{z_{i,i+1}^4:\text{$i\in\Z_n$ and $m_{i,i+1}\ge3$}\}
    \]
    is isomorphic to a cycle RAAG~\cite{jankiewicz2022right}*{Theorem~1.1}, against Lemma~\ref{lem:raagSubGrpOfDihedral}.
\end{proof}

\section{Embedding cycle Artin groups}
\label{sec:embeddCycArtGrp}

We study cycle Artin subgroups of two-dimensional Artin groups of hyperbolic type from a geometric point of view. We shall see that each cycle Artin subgroup corresponds to a cycle of standard trees in the Deligne complex.

\begin{defi}[cycle Artin groups]
    Let $C_n$ be a labelled cycle; we call the group $A_{C_n}$ a \emph{cycle Artin group}.     If $A_{C_n}$ is right angled, then we denote it by $R_{C_n}$ and call it a \emph{cycle RAAG}.
\end{defi}

For the sake of light notation, we never specify the edge-labelling of $C_n$ but we should always think of it as a labelled graph.\par
Adapted to the language of embeddings of cycle Artin groups, we will only use Proposition~\ref{prop:2d_iso_inv} in the form of its immediate corollary:

\begin{lem}
    \label{lem:no Z^3}
    Let $A_\Gamma$ be a two-dimensional Artin group of hyperbolic type; there does not exist any group monomorphism $R_{C_3}\to A_\Gamma$.
\end{lem}

Given a cycle Artin group embedding $A_{C_n}\to A_\Gamma$, we study the subcomplex of $D_\Gamma$ spanned by the minset of the standard generators of $A_{C_n}$ in $A_\Gamma$. Our goal (Lemma~\ref{lem:mono->cycle}) is to show that such subcomplex is a cycle of standard trees, which we now define for a two-dimensional Artin group $A_\Gamma$. Recall that a standard tree is, up to translation, the fixed-point set of a standard generator of $A_\Gamma$.

\begin{defi}[cycles of standard trees]
    \label{def:cycStdTree}
    Let $n\in\N_{\ge3}$; a \emph{cycle of standard trees} in $D_\Gamma$ is a finite set of distinct standard trees $\{T_i\}_{i\in\Z_n}$ such that, for every $i\in\Z_n$:
    \begin{enumerate}
        \item  $T_i\cap T_{i+1}\ne\varnothing$;
        \item $T_{i-1}\cap T_i\cap T_{i+1}=\varnothing$. 
    \end{enumerate} 
\end{defi}

    For every $i\in\Z_n$, let us denote by $v_i$ the intersection vertex $T_i\cap T_{i+1}$ (which is necessarily unique~\cite{martin2022acylindrical}*{Lemma 4.3}). A cycle of standard trees $\{T_i\}_{i\in\Z_n}$ is \emph {simple} if the loop of $D_\Gamma$ that is obtained by concatenating the geodesic segments $[v_i,v_{i+1}]$ is an embedded loop. \par
    A \emph{sub-cycle} of $\{T_i\}_{i\in\Z_n}$ is an injective map $\alpha\colon\Z_k\to\Z_n$ such that $\{T_{\alpha(j)}\}_{j\in\Z_k}$ is a cycle of standard trees.

\begin{rem}
    \label{rem:cycStdTrsOfBMV}
    The term cycle of standard trees (alongside loop of standard trees) also appears in work of Blufstein-Martin-Vaskou~\cite{blufstein2024homomorphisms}*{Definition 6.2}, but our notion is weaker. In particular, we do not require that, for every distinct indices $i,j\in\Z_n$, if $j\ne i\pm1$, then the subgroup $\Span{g,h}_{A_\Gamma}$ is non-abelian free, where $T_i=\fix g$ and $T_j = \fix h$.
\end{rem}
    
We remark that every cycle of standard trees admits a sub-cycle that is simple.

\begin{lem}
    \label{lem:simSubCyc}
    Let $\mathscr L = \{T_i\}_{i\in\Z_n}$ be a cycle of standard trees that is not simple. There are $j\in\{1,\dots,n\}$ and $k\in\{2,\dots,n-1\}$ such that $\mathscr M= \{T_j,T_{j+1},\dots, T_{j+k}\}$ is a simple sub-cycle of $\mathscr L$.
\end{lem}

\begin{proof}
    For every $i\in\Z_n$, let $v_i\coloneqq T_i\cap T_{i+1}$ and let us denote by $\gamma_i$ the geodesic segment $[v_i,v_{i+1}]$. Let $m$ be minimal with the property that $\gamma_1\cup\cdots\cup\gamma_{m-1}$ is an embedded loop and notice that, because $\mathscr L$ is not simple, $m<n$. Let $j\in\{1,\dots,m-1\}$ be maximal with the property that $\gamma_m\cap\gamma_j\ne\varnothing$.  Let us set $k=m-j$, modulo~$n$. By construction, the cycle $\{T_j,T_{j+1}\dots,T_{j+k}\}$ is simple.
\end{proof}

\subsection{The general case}\label{subsec:cycStdTrsFromCycArt} Throughout the subsection we work in the following setting. 

\begin{setting}
    \label{set:setting}
    Let $A_\Gamma$ be a two-dimensional Artin group of hyperbolic type, let $n\in\N_{\ge3}$ and let $\phi\colon A_{C_n}\to A_\Gamma$ be a group monomoprhism. Being $\{u_i\}_{i\in\Z_n}$ the set of vertices of $C_n$,  we set $y_i\coloneq \phi(u_i)$. For the sake of light notation, we denote by $m_{i,i+1}$ the label of the edge $\{u_i,u_{i+1}\}$.
\end{setting}

\begin{lem}
    Let us work under the hypotheses of Setting~\ref{set:setting}.  For every $i\in\Z_n$, $y_i$ acts elliptically on $D_\Gamma$.
\end{lem}

\begin{proof}
    Up to a shift of the indices, let us assume by contradiction that $y_2$ acts loxodromically on $D_\Gamma$. We shall distinguish two cases.\medskip 

    If at least one of the labels ${m_{12}}$ and $m_{23}$ is not a $2$, say $m_{12}\ge3$, then Lemma~\ref{lem:dihedral of hyp-type are parab} gives that $\Span {y_1,y_2}$ is a dihedral parabolic subgroup of $A_\Gamma$, that is, the stabiliser of a vertex of $D_\Gamma$, a contradiction.\par
    Let us assume that $m_{12}=m_{23}=2$. Since $\Span{y_1,y_2}\cong \Z^2$, it is contained in the centraliser of a conjugate of a standard generator of $A_\Gamma$ (see Theorem~\ref{thm:classify Z^2 subgrp}). Up to conjugation, we may assume that there is $a\in\ver\Gamma$ such that $\Span{y_1,y_2}\le \centr{A_\Gamma}{a}$. Because $y_2$ acts loxodromically on $D_\Gamma$ and stabilises the standard tree $T_a$, we have that $\axis {y_2} \subseteq T_a$. By the same argument applied to the subgroup $\Span{y_2,y_3}$, there exist $b\in\ver\Gamma$ and $g\in A_\Gamma$ such that $\Span{y_2,y_3}\le g\centr{A_\Gamma}{b}g^{-1}$ and $\axis{y_2}\subseteq g\cdot  T_b$. Because distinct standard trees intersect in at most one vertex~\cite{martin2022acylindrical}*{Remark~4.4}, we must have that the standard trees $T_a$ and $g\cdot T_b$ coincide and thus $\Span{y_1,y_2,y_3}\le\centr{A_\Gamma}{a}$. By means of the isomorphism $\centr{A_\Gamma}{a}\cong\Span a\times F$ (for some finitely generated free group $F$), for $i\in\{1,2,3\}$, let us write $y_i=(a^{n_i},x_i)$. Because $y_1$ and $y_2$ commute, we must have that $x_1$ and $x_2$ commute, meaning that there exists $z'\in F\setminus\{1\}$ such that $x_1,x_2\in\Span {z'}$, since any abelian subgroup of $F$ must be cyclic. Similarly, because $y_2$ and $y_3$ commute, there exists $z''\in F\setminus\{1\}$ such that $x_3\in\Span {z''}$. In particular, there exists $z\in F$ such that $x_1,x_2,x_3\in\Span z$. It follows that $y_1$ and $y_3$ commute, against Lemma~\ref{lem:no Z^3}.
\end{proof}

\begin{lem}
    \label{lem:intersection fps}
    Let us work under the hypotheses of Setting~\ref{set:setting}. For every $i\in\Z_n$, $\fix{y_i}$ and $\fix{y_{i+1}}$ intersect.
\end{lem}

\begin{proof}
    Because $\Span{y_i,y_{i+1}}\cong\da m$ for some $m\in\N_{\ge2}$, it is either contained in a dihedral parabolic subgroup of $A_\Gamma$ or, when $m=2$, in the commutator of the conjugate of a standard generator (see Lemma~\ref{lem:dihedral of hyp-type are parab} and Theorem~\ref{thm:classify Z^2 subgrp}). In the first case, $y_i$ and $y_{i+1}$ both fix a vertex of type~$2$. In the second case, up to conjugation, we can assume that there exists $a\in\ver\Gamma$ such that $\Span{y_{i},y_{i+1}}\le\centr {A_\Gamma} a\cong\Span{a}\times F$. If one between $y_i$ and $y_{i+1}$ is tree-elliptic, then we are done; if they are both vertex-elliptic, then they must fix the same vertex, for, if not, then a standard ping-pong argument in the standard trees associated with $a$ shows that they would generate a non-abelian free subgroup.
\end{proof}

\begin{lem}
    \label{lem:distinct std trees}
    Let us work under the hypotheses of Setting~\ref{set:setting}. For every $i,j\in\Z_n$, if ${y_i}$ and ${y_{j}}$ act tree-elliptically on $D_\Gamma$ fixing the same standard tree, then $i=j$. 
\end{lem}

\begin{proof}
    Up to translation in $D_\Gamma$, we may assume that $y_i$ and $y_j$ fix the standard tree $T_a$, for some $a\in\ver\Gamma$. In particular, we have that $\Span{y_i,y_j}\le\Span a$, which occurs if and only if $i=j$.
\end{proof}

\begin{lem}
    \label{lem:cycle of std trees}
    Let us work under the hypotheses of Setting~\ref{set:setting}. Let $k=\abs{\{i\in\Z_n:\text{$y_i$ acts tree-elliptically on $D_\Gamma$}\}}$. Then $k\ge3$ and there are $k$ distinct standard trees $T_1,\dots,T_k$ of $D_\Gamma$ such that 
    \[
    \bigcup_{i\in\Z_n}\fix{y_i}=\bigcup_{j=1}^kT_j.
    \]
\end{lem}

\begin{proof}
    Let us show that $k\ge3$ by showing that the cases $k\in\{0,1,2\}$ are not admissible. Note that the case $A_{C_n}=\Z^3$ is already ruled out by Lemma~\ref{lem:no Z^3}. Moreover, as soon as we prove that the image of $\phi$ is contained in a dihedral subgroup, we obtain a contradiction, using Corollary~\ref{cor:cycle subgrp dihedral}.\medskip
    
    If $k=0$, then it means that all the $y_i$'s act vertex-elliptically on $D_\Gamma$. By Lemma~\ref{lem:intersection fps}, all the $y_i$'s must fix the same vertex, that is, the image of $\phi$ is contained in a dihedral subgroup of $A_\Gamma$. \par 
    If $k=1$, then we may assume that $y_1$ is the only tree-elliptic element. By Lemma~\ref{lem:intersection fps}, the elements $y_2,\dots,y_n$ fix the same vertex of type~$2$ that lies in the standard tree $\fix {y_1}$. In particular, the image of $\phi$ is contained in a dihedral parabolic subgroup, a contradiction. \par
    Let us assume $k=2$. Up to a shift of the indices, we may assume that $T_1=\fix {y_1}$ and let us denote $v\coloneqq T_1\cap T_2$, which is exactly one vertex, because two distinct standard trees intersect in at most one vertex~\cite{martin2022acylindrical}*{Remark 4.4}. Let us assume by contradiction that $\fix{y_2}$ is neither $v$ nor $T_2$. By Lemma~\ref{lem:distinct std trees}, $y_2$ acts vertex-elliptically, hence fixing a vertex of $T_1$ that is distinct from $v$. It then follows that there is $i\in\{3,\dots,n-1\}$ such that $\fix {y_{i}}=T_1$, against Lemma~\ref{lem:distinct std trees}. By repeating the same argument for $T_2$, it follows that all the $y_i$'s must fix the vertex~$v$, that is, the image of $\phi$ is contained in a dihedral Artin group. \medskip 
    
    In order to show the containment, it is sufficient to make the following observation: if $y_i$ is tree-elliptic, then $\fix {y_i}$ is one of the $T_j$'s; if $y_i$ is vertex-elliptic, then $\fix {y_i}\subseteq \fix {y_{i-1}}$.
\end{proof}

\begin{lem}
    \label{lem:mono->cycle}
    Let us work under the hypotheses of Setting~\ref{set:setting}. Let $A_\Gamma$ be a two-dimensional Artin group of hyperbolic type, let $n\in\N_{\ge3}$ and let $\phi\colon A_{C_n}\to A_\Gamma$ be a monomorphism of groups. Let $k$ and the $T_i$'s be as in the statement of Lemma~\ref{lem:cycle of std trees}. Then $\{T_i\}_{i\in\Z_k}$ is a cycle of standard trees.
\end{lem}

\begin{proof}
    Let $i\in\Z_n$. Lemma~\ref{lem:intersection fps} gives that $T_i\cap T_{i+1}\ne\varnothing$, while Lemma~\ref{lem:raagSubGrpOfDihedral} gives that $T_{i-1}\cap T_i\cap T_{i+1}=\varnothing$.
\end{proof}

\subsection{The RAAG case}

We turn our attention to embeddings of cycle RAAGs. These are more rigid than embeddings of cycle Artin groups, in the following sense.

\begin{lem}
    \label{lem:cycStdTrsAssToRaag}
    Let $A_\Gamma$ be a two-dimensional Artin group of hyperbolic type, let $n\in\N_{\ge4}$ and let $\phi\colon R_{C_n}\to A_\Gamma$ be a monomorphism of groups. Let $i\in\Z_n$.
    \begin{enumerate}
        \item If both $\phi(u_i)$ and $\phi(u_{i+1})$ act tree-elliptically on $D_\Gamma$, then $\fix {\phi(u_i)}\cap \fix {\phi(u_{i+1})}$ consists of a type-two vertex of label~$2$.
        \item If $\phi(u_i)$ acts tree-elliptically and $\phi(u_{i+1})$ acts vertex-elliptically on $D_\Gamma$, then $\fix {\phi(u_{i+1})}$ consists of a type-two vertex of label at least~$3$ and $\phi(u_{i+2})$ acts tree-elliptically on $D_\Gamma$.
        \item The elements $\phi(u_i)$ and $\phi(u_{i+1})$ cannot act both vertex-elliptically on $D_\Gamma$.
    \end{enumerate}
\end{lem}

\begin{proof}
    For $j\in\Z_n$, let $y_j=\phi(u_j)$. It is not restrictive to assume that $i=1$. By Lemma~\ref{lem:intersection fps} and Lemma~\ref{lem:distinct std trees} $\fix {y_1}$ and $\fix {y_2}$ intersect in a type-two vertex $v\coloneqq g\Span {a,b}$, for some $\{a,b\}\in\edge\Gamma$ and for some $g\in A_\Gamma$. \medskip 
    
    Let us assume that both $y_1$ and $y_2$ act tree-elliptically. Because $y_1$ and $y_2$ fix~$v$, we have that $y_1,y_2\in g\Span{a,b}g^{-1}$. It follows that there exists a set of vertices $\{x_1,x_2\}=\{a,b\}$ and two elements $h_1,h_2\in\Span {a,b}$ such that $y_1=(gh_1)x_1(gh_1)^{-1}$ and $y_2=(gh_2)x_2(gh_2)^{-1}$~\cite{blufstein2023parabolic}*{Theorem 1.1}. Because $\Span {y_1,y_2}\cong \Z^2$, the only possibility is that $m_{ab}=2$~\cite{blufstein2024homomorphisms}*{Lemma 2.12}.\medskip 

    Let us now analyse the case of $y_1$ acting tree-elliptically and $y_2$ acting vertex-elliptically. By Lemma~\ref{lem:intersection fps}, the elements $y_1,y_2,y_3$ all fix $v$ and thus $\Span{y_1,y_3}\le\stab{A_\Gamma}{v}$, in particular. Because $\Span{y_1,y_3}\cong F_2$, it follows that $\stab{A_\Gamma}{v}$ cannot be abelian, that is, $m_{ab}\ge3$.\par
    Let us further assume by contradiction that $y_3$ acts vertex-elliptically on $D_\Gamma$, hence fixing~$v$. Again by Lemma~\ref{lem:intersection fps}, it follows that $y_4$ fixes $v$ as well and thus $\Span{y_1,\dots,y_4}\le\stab{A_\Gamma}{v}\cong\da{m_{ab}}$. If $n=4$, then $\Span{y_1,\dots,y_4}\cong R_{C_4}$; if $n\ge 5$, then $\Span{y_1,\dots,y_4}\cong R_{P_4}$: in both cases, Lemma~\ref{lem:raagSubGrpOfDihedral} yields a contradiction.\medskip

    Finally, let us assume that $y_1$ and $y_2$ act both vertex-elliptically. Let $\ell$ be the minimum of the set
    \[
    \{j\in\{3,\dots,n\}:\text{$y_j$ acts tree-elliptically on $D_\Gamma$}\},
    \]
    whose existence is granted by Lemma~\ref{lem:cycle of std trees}. A slight modification of the argument of the previous point applied to the elements $y_{\ell-2},y_{\ell-1}$ and $y_\ell$, yields a contradiction.
\end{proof}

\section{Cycles of standard trees and disc diagrams in the Deligne complex}
\label{sec:comb gbt}

Let $A_\Gamma$ be a two-dimensional Artin group of hyperbolic type. We saw in Section~\ref{sec:embeddCycArtGrp} that embeddings of cycle Artin groups in $A_\Gamma$ correspond to cycles of standard trees in $D_\Gamma$. We now associate a combinatorial structure to cycles of standard trees, that is, the one of a \emph{disc diagram}.

\subsection{Combinatorial Gauss--Bonnet theorem} We recall the notion of disc diagram and the combinatorial version of the Gauss--Bonnet theorem.

\begin{defi}[disc diagrams]
    Let $D$ be a two-dimensional finite contractible polygonal complex that embeds into a disc and let $X$ be a polygonal complex.
    \begin{enumerate}
        \item A \emph{disc diagram} is a combinatorial map $D\to X$;
        \item If $D$ is homeomorphic to a disc, then $D\to X$ is said to be a \emph{non-singular} disc diagram. In this case, such homeomorphism gives a definition of boundary $\partial D$ and interior $\operatorname{int}(D)$ of the disc diagram;
        \item A disc diagram $D\to X$ is \emph{reduced} if distinct $2$-cells of $D$ that share an edge get mapped to distinct $2$-cells of $X$;
        \item a disc diagram $D\to X$ is \emph{non-degenerate} if it is injective in every $2$-cell.
    \end{enumerate} 
\end{defi}

For a $2$-cell $P$ of $D$ and a vertex $v\in P$, the pair $(v,P)$ is said to be the \emph{corner} of~$P$ at $v$. For a vertex $v$ of $D$ we denote by $\corner v$ the set of corners at $v$; for a $2$-cell $P$ of~$D$ we denote by $\corner P$ the set of corners of $P$. We denote by $\corner D$ the set of corners of $D$.

\begin{defi}[curvature]
    \label{def:curvature}
    Let $D$ be a disc diagram and let $\angle\colon\corner D\to\R_{\ge0}$ be a map, which we call an \emph{assignment of angles}.
    \begin{enumerate}
        \item For a $2$-cell $P$ whose boundary consists of $n$ $1$-cells, the \emph{curvature} of $P$ is 
        \[
        \kappa(P)\coloneqq (2-n)\pi+\sum_{c\in \corner{P}}\angle (c)
        \]
        \item For a vertex $v$ contained in the boundary $\partial D$, the \emph{curvature} of $v$ is
        \[
        \kappa(v)\coloneqq \pi-\sum_{c\in\corner v}\angle(c)
        \]
        \item For a vertex $v$ contained in the interior of $D$, the \emph{curvature} of $v$ is
        \[
        \kappa(v)\coloneqq 2\pi-\sum_{c\in\corner v}\angle (c)
        \]
    \end{enumerate}
\end{defi}

In our proofs we will consider disc diagrams on the Deligne complex. Depending on whether we consider the simplicial or the cubical version, we obtain two different assignments of angles.

\begin{example}[Moussong angles]\label{ex:moussongAngles}
    Let $A_\Gamma$ be a two-dimensional Artin group. We think of the simplicial Deligne complex $D_\Gamma$ endowed with the Moussong metric. For a disc diagram $F\colon D\to D_\Gamma$, to each corner $(v,P)$ of $D$ we assign the angle $\angle(v,P)\coloneqq\angle(F(v),F(P))$. We refer to Figure~\ref{fig:deligneTriangle} for the angles on the triangles of~$D_\Gamma$.
\end{example}

\begin{example}[cubical angles]\label{ex:cubicalAngles}
    If $A_\Gamma$ is a two dimensional and triangle free, then we endow the cubical Deligne complex $C_\Gamma$ with the cubical metric (that is, all $2$-cubes are euclidean squares). For a disc diagram $G\colon E\to C_\Gamma$, to each corner $(v,P)$ of $E$ we assign the angle $\angle(v,P)\coloneqq\angle(G(v),G(P))=\frac\pi2$.
\end{example}

\begin{thm}[combinatorial Gauss--Bonnet theorem,\cite{ballmann1996nonpositively}]
    \label{thm:comb gbt}
    Let $D$ be a disc diagram with an assignment of angles $\angle\colon \corner D\to \R_{\ge 0}$. Then
    \[
    \sum_{v\in D^{(0)}}\kappa (v)+\sum_{P\in D^{(2)}}\kappa(P)=2\pi.
    \]
\end{thm}

\subsection{A disc diagram from a cycle of standard trees}
\label{subsec:dd loop std trees}

Let $A_\Gamma$ be a two-dimensional Artin group.  Let $\{T_i\}_{i\in\Z_n}\subseteq D_\Gamma$ be a cycle of standard trees and, possibly after passing to a sub-cycle, we may assume that it is simple. For every $i\in\Z_n$, let $v_i$ be the vertex in which the standard trees $T_i$ and $T_{i+1}$ intersect and let $\sigma_i\subseteq D_\Gamma$ be the geodesic segment connecting $v_i$ with $v_{i+1}$. Because standard trees are convex subsets of $D_\Gamma$, $\sigma_i$ is contained in $T_{i+1}$. Denote $\sigma\coloneqq\bigcup_{i\in\Z_n}\sigma_i$, which is a connected subset. Because $D_\Gamma$ is contractible~\cite{charney1995k}, there exists a non-singular, non-degenerate disc diagram $D\to D_\Gamma$ that maps $\partial{D}$ isomorphically to $\sigma$~\cite{januszkiewicz2006simplicial}*{Lemma 1.6}. We call every such disc diagram a \emph{filling diagram} for $\sigma$. We may always assume that $D\to D_\Gamma$ is a \emph{minimal} filling diagram for $\sigma$, that is, it contains the minimum number of $2$-cells amongst the filling diagrams for $\sigma$. The minimality implies that $D\to D_\Gamma$ is reduced, for otherwise a standard argument on surging $2$-cells would produce a diagram with strictly fewer $2$-cells. \par 
By means of $D\to D_\Gamma$ we can talk about types of vertices of $D$, as in Definition~\ref{def:type of vertices}. \par

To exploit the combinatorial Gauss--Bonnet theorem in~$D$, we need to assign angles to the corners of~$D$. Let us fix a minimal filling diagram $F\colon D\to D_\Gamma$ for $\sigma$; for every corner $(v,P)$ of $D$, we define $\angle (v,P)\coloneqq\angle(F(v),F(P))$ (see Example~\ref{ex:moussongAngles}). The equality from Theorem~\ref{thm:comb gbt} becomes
\[
2\pi=\sum_{i=1}^n\kappa(v_i)+\sum_{u\in{D^{(0)}}}\kappa(u)+\sum_{P\in D^{(2)}}\kappa(P),
\]
where the $u$'s range amongst all the vertices of $D$ that are not one of the intersection vertices $v_i$'s. The choice to break down the sum of curvatures of vertices into two sums is made to emphasise that, in our settings, the intersection vertices $v_i$'s will be the ones that will possibly contribute with positive curvature. \par

\begin{rem}[vertices of type~$0$]\label{rem:verTyp0}
    Because distinct standard trees intersect in at most one vertex~\cite{hagen2024extra}*{Corollary 2.18}, $D$ contains at least a vertex $v_0$ of type~$0$, for otherwise $D$ would be contained in the $1$-skeleton of $D_\Gamma$, thus not homeomorphic to a $2$-disc.
\end{rem}

\begin{rem}
    In $D_\Gamma$ and $C_\Gamma$, the \emph{essential $1$-skeletons}, i.e. the subcomplexes spanned by the vertices of type~$1$ and~$2$, coincide. Because standard trees are contained in the essential $1$-skeleton, it is straightforward to check that the above construction extends likewise to cycles of standard trees in $C_\Gamma$ (in which case, the assignment of angles is the one described in Example~\ref{ex:cubicalAngles}). In each proof, we will freely choose the most convenient complex to work with.
\end{rem}

\subsection{Computation of curvatures}\label{subsec:compOfCurv} We end this section with an estimate of curvatures relative to the Moussong metric on $D_\Gamma$ and the cubical metric on $C_\Gamma$.

\begin{lem}
    \label{lem:reduced dgr -> immersion}
    Let $F\colon D\to X$ be a reduced disc diagram. For every $v\in \ver D$, the restriction $F\colon \link{D}{v}\to\link{X}{F(v)}$ is an immersion (i.e. is locally injective). In particular, 
    \[
    \girth{\link D v}\ge\girth{\link X {F(v)}}.
    \]
    Moreover, if $X=D_\Gamma$ is the Deligne complex of some two-dimensional Artin group and $F\colon D\to D_\Gamma$ is obtained with the construction described in Section~\ref{subsec:dd loop std trees}, then
    \[
    \wg{\link D v}\ge\wg{\link {D_\Gamma} {F(v)}}.
    \]
\end{lem}

\begin{proof}
    The fact that the restriction $F\colon\link D v\to\link X{F(v)}$ is an immersion follows from the fact that the disc diagram $D\to X$ is reduced together with the observation that edges of $\link D v$ correspond to $2$-cells of $D$ containing $v$: if two adjacent edges of $\link D v$ had the same image in $\link X {F(v)}$, then the corresponding two $2$-cells would get mapped to the same $2$-cell of $X$, contradicting the fact that $F$ is reduced. \par 
    To show that $F$ lowers the girth, it is sufficient to notice that non-backtracking paths in $\link D v$ get mapped to non-backtracking paths of $\link X {F((v)}$, because $F\colon D\to X$ is a reduced disc diagram.\par
    The ``moreover'' claim follows by combining the previous one and the fact that the type and label of the vertices of $D$ is the one induced by $F\colon D\to D_\Gamma$.
\end{proof}

Let $A_\Gamma$ be a two-dimensional Artin group and let $\{T_i\}_{i\in\Z_n}$ be a simple cycle of standard trees in $D_\Gamma$. Recall that we denote by $v_i$ the intersection vertex $T_i\cap T_{i+1}$ and by $\sigma$ the loop obtained by concatenating the segments $[v_i,v_{i+1}]$'s. 

\begin{lem}[computing curvatures]
    \label{lem:cpt curvatures}
    Let $F\colon D\to D_\Gamma$ be a minimal filling diagram for~$\sigma$. Every $2$-cell of $D$ has non-positive curvature. Moreover, for every vertex $v\in\ver D$:
    \begin{itemize}
        \item if $v$ is of type~$0$, then $v$ belongs to the interior of $D$ and there exists a cycle subgraph $\gamma\subseteq\Gamma$ such that
        \[
        \kappa(v)\le\pi(2-\girth\gamma)+\pi\sum_{e\in\edge\gamma}\frac{1}{m_e};
        \]
        \item if $v$ is either a vertex of type~$1$ or a vertex of type~$2$ that is not one of the intersection vertices $v_i$'s, then $\kappa(v)\le0$.
    \end{itemize}
    In particular, if $A_\Gamma$ is of hyperbolic type, then vertices of type~$0$ have strictly negative curvature. 
\end{lem}

\begin{proof}
    Because $2$-simplices of $D_\Gamma$ are euclidean triangles, every $2$-cell of $D$ has non-positive curvature.\par
    Let $v$ be a vertex in the interior of $D$.  Since we are using the angular metric in $\link D v$, the quantity $\sum_{(v,P)\in\corner v}\angle(v,P)$ (i.e. the curvature of $v$) corresponds to the length of a loop in $\link D v$, say $\lambda$. By Lemma~\ref{lem:reduced dgr -> immersion}, $F(\lambda)$ is a loop (not necessarily simple) in $\link {D_\Gamma} {F(v)}$ of shorter length. Therefore $\kappa(v)\le\kappa(F(v))$. Let us then find an upper-bound for $\kappa(F(v))$ or, in other words, let us find a lower-bound for the length of a loop in $\link {D_\Gamma} {F(v)}$.\par

    First we consider the case $v$ lies in the interior of $D$.
    If $v$ is of type~$2$, then $\link{D_\Gamma}{F(v)}$ has girth $2\pi$~\cite{appel1983artin}*{Lemma~6}.
    If $v$ is of type~$1$, then $\link{D_\Gamma}{F(v)}$ is isomorphic to the join $\Z*\{1,\dots,k\}$ of two edgeless graphs, every edge being of length $\frac\pi2$, in which case it has girth $2\pi$ (unless $k=1$, in which case it is a tree).  
    If $v$ is of type~$0$, then we can identify $\link {D_\Gamma} {F(v)}$ with (the barycentric subdivision of) $\Gamma$ and the length of a cycle subgraph $\gamma\subseteq\Gamma$ is 
    \[
    \operatorname{len}_{\Gamma}(\gamma)=\sum_{e\in\edge\gamma}\frac{m_e-1}{m_e}\pi=\pi\girth\gamma-\sum_{e\in\edge\gamma}\frac{\pi}{m_e}.
    \]
    Note that the equality $\operatorname{len}_{\Gamma}(\gamma)=2\pi$ occurs only if $\gamma$ is a $(p,q,r)$-triangle with $\frac1p+\frac1q+\frac1r=1$ or $\gamma$ is the $(2,2,2,2)$-square. In particular, when $A_\Gamma$ is of hyperbolic type, $\operatorname{len}_{\Gamma}(\gamma)>2\pi$.\par
    Let now $v$ be a boundary vertex that is not one of the $v_i$'s.  Because $D$ is a filling diagram for $\sigma$ and standard trees do not contain vertices of type~$0$, $v$ can only be of type~$1$ or~$2$. In this case, the quantity $\sum_{(v,P)\in\corner v}\angle(v,P)$ corresponds to the length of a path in $\link D v$ between (the two vertices corresponding to) the two edges of the standard tree in which~$v$ lies that are incident to~$v$. Such path has length at least~$\pi$~\cite{martin2022acylindrical}*{Corollary 4.2}.
\end{proof}

The counterpart for the cubical Deligne complex with the cubical metric works likewise.

\begin{lem}[computing cubical curvatures]
    \label{lem:cpt cube curvatures}
    Let us assume that $\girth\Gamma\ge4$ and let $F\colon D\to C_\Gamma$ be a minimal filling diagram for~$\sigma$. Every $2$-cell of $D$ has non-positive curvature. Moreover, for every vertex $v\in\ver D$:
    \begin{itemize}
        \item if $v$ is of type~$0$, then $v$ belongs to the interior of $D$ and there exists a cycle subgraph $\gamma\subseteq\Gamma$ such that
        \[
        \kappa(v)\le(4-{\girth\gamma})\frac \pi 2;
        \]
        \item if $v$ is either a vertex of type~$1$ or a vertex of type~$2$ that is not one of the intersection vertices $v_i$'s, then $\kappa(v)\le0$.
    \end{itemize}
\end{lem}

\begin{proof}
    The proof follows the lines of the one of Lemma~\ref{lem:cpt curvatures}. Note that $\girth\Gamma\ge4$ ensures that $2$-cells of $D$ have non-positive curvature.\medskip

    Let $v$ be a vertex in the interior of $D$. Using the same notation used in the previous proof, we shall find a lower-bound for the girth of $\link{C_\Gamma}{F(v)}$. If $v$ is of type~$2$, then $\link{C_\Gamma}{F(v)}$ has girth at least $2\pi$~\cite{appel1983artin}*{Lemma~6}. If $v$ is of type~$1$, then $\link {C_{\Gamma}}{F(v)}$ is a join $\Z*\{1,\dots,k\}$, where each edge has length $\frac\pi2$. If $v$ is of type~$2$, then $\link{C_\Gamma}{F(v)}$ can be identified with the barycentric subdivision of $\Gamma$ and a loop $\gamma\subseteq\Gamma$ has length
    \[
    \operatorname{len}_\Gamma(\gamma)=\girth\gamma\cdot\frac\pi2.
    \]
    Because $\girth\Gamma\ge4$, $\operatorname{len}_{\Gamma}(\gamma)\ge2\pi$.\medskip

    If $v$ is a vertex in the boundary of $D$ that is not one of the $v_i$'s, then $v$ is necessarily of type~$1$ or~$2$. Again, a path in $\link{C_\Gamma}{F(v)}$ between two edges of the same standard tree has length at least~$\pi$~\cite{martin2022acylindrical}*{Corollary 4.2}.
\end{proof}

\section{Length and weighted girth of a cycle of standard trees}
\label{sec:lenAndWgCycStdTrs}

We recall the definition of girth of a graph and introduce a new graph invariant, the weighted girth. We then introduce two invariants for cycles of standard trees, namely the length and the weighted girth, and show that they admit a lower bound that only depends on the defining graph. We do so in Lemma~\ref{lem:lenCycStdTrs} (for the length) and Lemma~\ref{lem:wgCycStdTrs} (for the weighted girth).

\begin{rem}[tree defining graph]
    \label{rem:treeDefGrph}
    If $\Gamma$ is a tree, then there are no cycles of standard trees in~$D_\Gamma$. To see this, assume that a cycle of standard trees exists and construct a filling diagram $F\colon D\to D_\Gamma$. By Remark~\ref{rem:verTyp0}, $D$ contains a type-zero vertex $u$ in its interior. Because $D$ is homeomorphic to a disc, $\link D u$ is topologically a circle. By Lemma~\ref{lem:reduced dgr -> immersion}, the image of $\link D u$ under $F$ is a (possibly non simple) loop in $\link{D_\Gamma}{F(u)}$, while the latter is isomorphic to $\Gamma$, a contradiction.\par
    In particular, the statements of Lemma~\ref{lem:lenCycStdTrs} and Lemma~\ref{lem:wgCycStdTrs} vacuously hold true when $\Gamma$ is a tree.
\end{rem}

\subsection{The girth case}\label{subsec:girth} Recall that the girth of a simplicial graph is the length of a shortest cycle, with the convention that trees have girth~$\infty$.

\begin{example}
    \label{ex:canonCycStdTrees}
    The parabolic subgroups of $A_\Gamma$ that correspond to cycle subgraphs define cycles of standard trees in a canonical way. Let $\gamma\subseteq\Gamma$ be a simple cycle subgraph. Every $v\in\ver\gamma$ acts tree-elliptically on $D_\Gamma$ and the set $\{\fix v\}_{v\in\ver\gamma}$ is a simple cycle of $\girth\gamma$-many standard trees. Moreover, such standard trees bound the subcomplex $K_\gamma$ of the standard fundamental domain $K_\Gamma$.
\end{example}

\begin{lem}[length of cycle of standard trees]
    \label{lem:lenCycStdTrs}
    Let $A_\Gamma$ be a two-dimensional Artin group and let $\{T_i\}_{i\in\Z_k}$ be a cycle of standard trees of the Deligne complex. Then $k\ge\girth\Gamma$.\par
    Moreover, if $n\coloneqq\girth\Gamma\ge5$, then, for every cycle of $n$ standard trees $\{T_i\}_{i\in\Z_n}$, there exists a cycle subgraph $\gamma\subseteq\Gamma$ such that $\girth\gamma=n$ and the $T_i$'s bound an $A_\Gamma$-translate of $K_\gamma$.
\end{lem}

\begin{proof}
    Up to passing to a sub-cycle with fewer trees, it is not restrictive to assume that $\{T_i\}_{i\in\Z_k}$ is simple. By definition, every cycle of standard trees consists of at least three trees, thus we can assume that $\girth\Gamma\ge4$ and work with the cubical Deligne complex $C_\Gamma$. Let us consider a minimal filling disc diagram $D$ as constructed in Section~\ref{sec:comb gbt}. For $i\in\Z_k$, let $v_i$ be the intersection vertex $T_{i-1}\cap T_i$ and let $v_0$ be a type-zero vertex, which necessarily exists and belongs to the interior of $D$ (see Remark~\ref{rem:verTyp0}). The combinatorial Gauss--Bonnet theorem, together with the fact that the remaining vertices contribute with non-positive curvature, gives
    \[
    2\pi\le\sum_{i=1}^k\kappa(v_i)+\kappa(v_0)\le k\frac \pi 2+(4-\girth\gamma)\frac\pi2,
    \]
    where $\gamma$ is a cycle subgraph of $\Gamma$ (we are invoking Lemma~\ref{lem:cpt cube curvatures}). The claim follows by noticing that $\girth\Gamma$ is the minimum girth of a cycle subgraph of $\Gamma$.\medskip

    The proof of the second statement follows by a more careful examination of the curvatures. By the first part of the lemma (together with Example~\ref{ex:canonCycStdTrees}), $n$ is minimal, implying that the cycle $\{T_i\}_{i\in\Z_n}$ is simple. The combinatorial Gauss--Bonnet theorem gives
    \[
    2\pi\le\sum_{i=1}^n\kappa(v_i)+\kappa(v_0)\le\sum_{i=1}^n\kappa(v_i)+(4-\girth\gamma)\frac\pi2,
    \]
    for some cycle subgraph $\gamma\subseteq\Gamma$. If there were $j\in\{1,\dots,n\}$ such that $\kappa(v_j)<\frac\pi2$, then we would obtain 
    \[
    2\pi<n\frac\pi2+(4-\girth\gamma)\frac\pi2\le n\frac\pi2+(4-n)\frac\pi2=2\pi,
    \]
    a contradiction. In particular, $\sum_{i=1}^n\kappa(v_i)=n\frac\pi2$. If we had $\girth\gamma>n$, then
    \[
    2\pi<n\frac\pi2+2\pi-n\frac\pi2=2\pi,
    \]
    another contradiction. Therefore $\girth\gamma=n$ and each inequality is in fact an equality. Because $\girth\Gamma\ge5$ implies that $A_\Gamma$ is of hyperbolic type, vertices of type~$0$ have strictly negative curvature. It follows that $v_0$ is the only vertex of $D$ of type~$0$ and the trees $T_i$'s bound a loop in $\link{C_\Gamma}{v_0}\cong \Gamma$.
\end{proof}

\subsection{The weighted girth case} We introduce the weighted girth of a labelled graph.

\begin{defi}[weighted girth]\label{def:wg}
    Let $\Gamma$ be a labelled graph:
\begin{enumerate}
    \item for every simple cycle subgraph $\gamma\subseteq\Gamma$, the \emph{weighted girth} of $\gamma$ is the quantity
    \[
    \wg\gamma=\girth\gamma+\abs{\{\{a,b\}\in\edge\gamma:m_{ab}\ge3\}};
    \]
    \item the \emph{weighted girth} of $\Gamma$ is defined as the minimum
    \[
    \wg\Gamma=\min\{\wg\gamma:\text{$\gamma\subseteq\Gamma$ cycle subgraph}\}.
    \]
\end{enumerate}
    If $\Gamma$ is a tree, then we set $\wg\Gamma=\infty$.
\end{defi}

Equivalently, $\wg\Gamma$ can be defined as the girth of the graph that is obtained by subdividing every edge of $\Gamma$ with label at least~$3$. \par 

Likewise we define the weighted girth of a cycle of standard trees.

\begin{defi}
    \label{def:wgCycStdTrs}
    Let $A_\Gamma$ be a two-dimensional Artin group and let $\mathscr L=\{T_i\}_{i\in\Z_k}$ be a cycle of standard trees of $D_\Gamma$; the \emph{weighted girth} of $\mathscr L$ is
    \[
        \wg{\mathscr L}=k+\abs{\{i\in\Z_k:\lab{T_i\cap T_{i+1}}\ge3\}}.
    \]
\end{defi}

We saw in Example~\ref{ex:canonCycStdTrees} that cycle subgraphs of the defining graph $\Gamma$ determine cycles of standard trees of the same length in the Deligne complex. The correspondence extends to the weighted girths. 

\begin{example}
    \label{ex:canonCycStdTreesWG}
    Let $\gamma\subseteq\Gamma$ be a simple cycle subgraph. The cycle of standard trees $\{\fix a\}_{a\in\ver\gamma}$ bounding the subcomplex $K_\gamma$ has weighted girth $\wg\gamma$. Indeed, it has length $\girth\gamma$ and intersection vertices of label at least~$3$ correspond to edges of $\gamma$ of label at least~$3$.
\end{example}

Our goal is to show the following lower bound:

\begin{lem}
    \label{lem:wgCycStdTrs}
    Let $A_\Gamma$ be a two-dimensional Artin group.\medskip
    
    \emph{Triangle-free case}. Assume that $\girth\Gamma\ge4$. For every cycle of standard trees $\mathscr L\subseteq D_\Gamma$, $\wg{\mathscr L}\ge\wg\Gamma$.\par 
    Moreover, if $\girth\Gamma\ge5$, then, for every cycle of standard trees $\mathscr L=\{T_i\}_{i\in\Z_k}$ with $\wg{\mathscr L}=\wg\Gamma$, there exists a cycle subgraph $\gamma\subseteq\Gamma$ such that $\wg\gamma=\wg\Gamma$ and the $T_i$'s bound an $A_\Gamma$-translate of $K_\gamma$.\medskip
    
    \emph{Hyperbolic-type case}. Assume that $A_\Gamma$ is of hyperbolic type. For every cycle of standard trees $\mathscr L\subseteq D_\Gamma$, $\wg{\mathscr L}\ge\wg\Gamma$.
\end{lem}

\subsubsection{Triangle-free case} We first deal with the case in which the defining graph is triangle free, where we can rely on the cubical Deligne complex being a \cat 0 cube complex. 
The core of the proof is to understand \emph{simple} cycles of standard trees:

\begin{lem}
    \label{lem:wgSimCycStdTrs}
    Let $A_\Gamma$ be a two-dimen\-sion\-al Artin group with $\girth\Gamma\ge4$. For every simple cycle of standard trees $\mathscr L\subseteq C_\Gamma$, $\wg{\mathscr L}\ge\wg\Gamma$.
\end{lem}

\begin{proof}
    Let $\L\subseteq C_\Gamma$ be a simple cycle of standard trees. For $i\in\Z_k$, let $v_i$ be the intersection vertex $T_{i-1}\cap T_i$ and let us consider a filling diagram $f\colon D\to C_\Gamma$ for~$\mathscr L$. \par 
    Note that, because all $2$-cells are squares, which have zero-curvature, the curvature is all concentrated on the vertices. In order to simplify the computations, we are going to  construct a new disc diagram $F$ starting from $D$. First of all, for every type-$0$ vertex~$u$, we remove~$u$ and replace $\operatorname{Star}_D(u)$ with a single polygon~$P_u$, attached on $\link D u$. This surgery is well defined, since standard trees are contained in the subcomplex of~$C_\Gamma$ spanned by the vertices of type~$1$ and~$2$. See Figure~\ref{fig:polygonStructure} for an intuition.
    
    \begin{figure}
        \centering
        \includegraphics[width=0.45\linewidth]{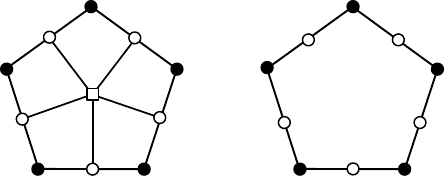}
        \caption{Removing a vertex of type~$0$ and replacing its star with a single polygon.}
        \label{fig:polygonStructure}
    \end{figure}

    This operation produces a new disc diagram, which, with an abuse of notation, we denote $f\colon E\to C_\Gamma$. The corners of $E$ are precisely those of the form $(v,P)$, where $P$ is a polygon of $E$ and $v$ is a vertex in $P$. To each corner we assign an angle~$\frac\pi2$.

    \begin{claim}
        \label{claim:EcellNonPosCurv}
        The $2$-cells of $E$ have curvature bounded above by $\frac\pi2(4-\girth\Gamma)$. In particular, they have non-positive curvature.
    \end{claim}

    \begin{proof}[Proof of Claim~\ref{claim:EcellNonPosCurv}]
        Let $P$ be a $2$-cell of $E$. The map $f$ sends the boundary $\partial P$ to a loop in $\link{C_\Gamma}u$, for some vertex~$u$ of type~$0$. Therefore $\partial P$ is a bipartite cycle consisting of $n$ vertices of type~$1$ and~$n$ vertices of type~$2$, for some $n\in\N_{\ge\girth\Gamma}$. It follows that
        \begin{align*}
            \kappa(P)=(2-2n)\pi+\sum_{i=1}^n\pi+\sum_{j=1}^n\frac\pi2=\frac\pi2(4-n) && \text{(Definition~\ref{def:curvature})}
        \end{align*}
        
        and the claim follows from the hypothesis $\girth\Gamma\ge4$.
    \end{proof}
    
    We now modify~$E$ to construct the disc diagram~$F$. For every $i\in\Z_k$, if $v_i$ has label $3$, then we consider $\epsilon\in\R_{>0}$ small enough such that the only vertex of $D$ that the closed ball $B$ of radius $\epsilon$ about $v_i$ contains is $v_i$ itself. For every edge $e$ of $E$ containing $v_i$, $F$ has a vertex $v_e$ at $e\cap\partial B$ (we call such vertices \emph{auxiliary vertices}); for every $2$-cell $P\in E^{(2)}$ containing $v_i$, $F$ has an edge at $P\cap\partial B$. Finally, we remove the interior of $B$. See Figure~\ref{fig:modify disc dgr} for an intuition. By construction, there is a bijection
    \begin{align*}
        E^{(2)}&\to F^{(2)}\\
        P&\mapsto\overline{P}
    \end{align*}
    between $2$-cells of $E$ and $2$-cells of $F$. If $a$ is an auxiliary vertex of $F$ and $\overline P$ is a $2$-cell containing it, then we assign $\angle(a,\overline P)=\frac\pi2$.

    \begin{figure}
        \centering
        \includegraphics[width=0.5\linewidth]{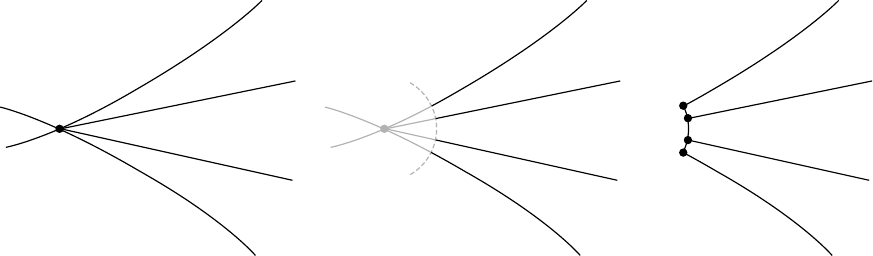}
        \caption{Truncating a corner of label at least $3$.}
        \label{fig:modify disc dgr}
    \end{figure}

        \begin{claim}
        \label{claim:FcellNonPosCurv}
        For every $2$-cell $P$ of $E$, $\kappa(\overline P)\le\kappa(P)$. In particular, the $2$-cells of $F$ have non-positive curvature.
    \end{claim}

    \begin{proof}[Proof of Claim~\ref{claim:FcellNonPosCurv}]
        Let $\overline P$ be a $2$-cell of $F$.  We denote by $e$ the number of edges of~$P$, by $v_1$ the number of vertices of $P$ of type~$1$, by $v_2$ the number of vertices of $P$ of type~$2$. Analogously we define $\overline e$, $\overline{v}_1$ and $\overline v_2$ for $\overline P$. Furthermore, we decompose $v_2$ as $v_2'+v_2''$, where $v_2''$ is the number of intersection vertices with label a least~$3$ (i.e. those vertices that got truncated). By the way we constructed $F$ from~$E$, we have that $\overline e = e+v_2''$, $\overline v_1 = v_1$ and $\overline v_2 = v_2'+2 v_2''$. Therefore
        \begin{align*}
            \kappa(\overline P) &= (2-\overline e)\pi + \overline v_1\pi+\overline v_2\frac{\pi}2   &\text{(Definition~\ref{def:curvature})}\\
            &=(2-e-v_2'')\pi + v_1\pi + (v_2'+2v_2'')\frac\pi2                                      &\text{(vertex truncation)}\\
            &= (2-e)\pi + v_1\pi + (v_2'+v_2'')\frac\pi2 - v_2''\frac\pi2                             &\text{(algebraic manipulation)}\\
            &=\kappa(P)-v_2''\frac\pi2                                                              &\text{(definition of curvature)}
        \end{align*}
        and the claim follows by recalling that $\kappa(P)$ is non-positive by Claim~\ref{claim:EcellNonPosCurv}.
    \end{proof}

    The number of vertices of $F$ that can possibly have positive curvature is at most $\wg{\mathscr L}$. Indeed, let us estimate the curvature of different kind of vertices. Vertices of type~$1$ and vertices of type~$2$ that are neither an intersection vertex nor an auxiliary vertex have non-positive curvature by Lemma~\ref{lem:cpt cube curvatures}. For each of the the intersection vertices $v_i$'s that got replaced, there are exactly two auxiliary vertices whose curvature is non-zero (its value is precisely $\frac\pi2$). Every vertex $v_i$ that did not get replaced by auxiliary vertices (i.e. those of label~$2$) can have curvature at most~$\frac\pi2$, depending on the number of $2$-cells that contain it. \par 
    In fact, some vertices carry strictly negative curvature.

    \begin{claim}
    \label{claim:curBndBy-pi/2}
        Vertices of type $2$ and label at least $3$ that did not correspond to the truncated vertices have strictly negative curvature, bounded above by $-\frac\pi2$.
    \end{claim}

    \begin{proof}[Proof of Claim~\ref{claim:curBndBy-pi/2}]
        Let $v$ be one such vertex, with label $m\ge3$. If $v$ belongs to the interior of $F$, then its link has combinatorial girth at least $2m$~\cite{appel1983artin}*{Lemma 6} and hence $\kappa(v)\le2\pi-2m\frac\pi2\le-\pi$. If $v$ belongs to the boundary of $F$, then a path in $\link F v$ between two vertices corresponding to edges of the same standard tree has combinatorial length at least $m$~\cite{martin2022acylindrical}*{Corollary 4.2}. It follows that $\kappa(v)\le\pi-m\frac\pi2\le-\frac\pi2$.  
    \end{proof}
    
     We shall exploit this negative curvature in order to artificially perform a ``redistribution''. Let $\overline{P}_0$ be a $2$-cell of $F$, whose existence is granted by the fact that $F$ is non-singular (since $D$ was). We can decompose the number of vertices of $\overline{P}_0$ as $p+q'+q''+a'+2a''$, where $p$ is the number of vertices of type~$1$, $q'$ is the number of vertices of type~$2$ and label $2$, $q''$ is the number of vertices of type $2$ and label at least $3$, $a'$ is the number of auxiliary vertices with zero curvature and $2a''$ is the number of auxiliary vertices with positive curvature (by construction they come in pairs). Let us set $\kappa'(\overline P_0)\coloneqq\kappa(\overline{P}_0)-\frac\pi2q''$. For every vertex $w\in \overline P_0$ of type $2$ and label at least $3$, let us set $\kappa'(w)\coloneqq\kappa(w)+\frac\pi2$. Curvatures of the remaining cells are not changed: for each other $2$-cell $\overline P\in F^{(2)}$, we set $\kappa'(\overline P)\coloneqq\kappa(\overline P)$; for each other vertex $v\in F^{(0)}$, we set $\kappa'(v)\coloneqq\kappa(v)$. By construction 
    \begin{equation}
        \label{eq:newCurvature}
        \sum_{v\in F^{(0)}}\kappa(v)+\sum_{\overline P\in F^{(2)}}\kappa(\overline P)=\sum_{v\in F^{(0)}}\kappa'(v)+\sum_{\overline P\in F^{(2)}}\kappa'(\overline P).
    \end{equation}

    \begin{claim}
        \label{claim:curvOfP0}
        We have $\kappa'(\overline{P}_0)\le2\pi-\frac\pi2\wg\Gamma$.
    \end{claim}

    \begin{proof}[Proof of Claim~\ref{claim:curvOfP0}]
        From the definition $\kappa'(\overline P_0)=\kappa(\overline{P}_0)-\frac\pi2q''$ we obtain
    \begin{align*}
        \kappa'(\overline{P}_0)&=\kappa(\overline{P}_0)-\frac\pi2q''\\  
            &=(2-p-q'-q''-a'-2a'')\pi+(p+a')\pi+(q'+q''+2a'')\frac\pi2 -q''\frac\pi2\\
            &=2\pi-\frac{\pi}{2}(q'+2q''+2a).
    \end{align*}
    Let $u_0$ be the type-$0$ vertex of $C_\Gamma$ such that $f(\partial P_0)\subseteq \link{C_\Gamma}{u_0}$. By means of the isomorphism $\link{C_\Gamma}{u_0}\cong\Gamma$, the quantity $q'+2q''+2a$ is bounded below by $\wg{\Gamma}$, so $\kappa'(\overline{P}_0)\le2\pi-\frac\pi2\wg\Gamma$.
    \end{proof}
    
    Vertices and $2$-cells that have non-positive $\kappa$-curvature still have non-positive $\kappa'$-curvature, by definition of $\kappa'$. In particular, all $2$-cells have non-positive $\kappa'$-curvature by Claim~\ref{claim:FcellNonPosCurv} and Claim~\ref{claim:curvOfP0}. The combinatorial Gauss--Bonnet theorem gives
    \begin{align}
    \label{eq:cgb for E}
    2\pi&=\sum_{v\in F^{(0)}}\kappa(v)+\sum_{\overline P\in F^{(2)}}\kappa(\overline P)&\text{(combinatorial Gauss--Bonnet)}\\
        &=\sum_{v\in F^{(0)}}\kappa'(v)+\sum_{\overline P\in F^{(2)}}\kappa'(\overline P)&\text{(Equation~\eqref{eq:newCurvature})}\nonumber\\
        &\le \wg {\mathscr L}\frac\pi2+\kappa'(\overline P_0)&\text{(truncation)}\nonumber\\
        &=\wg {\mathscr L}\frac\pi2+2\pi-\frac\pi2\wg{\Gamma}&\text{(Claim~\ref{claim:curvOfP0})}\nonumber
    \end{align}
    whence follows that $\wg {\mathscr L}\ge\wg\Gamma$.
\end{proof}

\begin{lem}
    \label{lem:wg simp < wg}
    Let $A_\Gamma$ be a two-dimensional Artin group with $\girth\Gamma\ge4$, let $\mathscr L$ be a cycle of standard trees of $D_\Gamma$. If $\mathscr L$ is not simple, then there exists a simple sub-cycle $\mathscr M$ of $\mathscr L$ with
    \[
    \wg{\mathscr M}<\wg{\mathscr L}.
    \]
\end{lem}

\begin{proof}
    \label{claim:simple_subcycle}
    Let $\mathscr L=\{T_i\}_{i\in\Z_n}$. Up to a shift of the indices, we can assume that $\mathscr M= \{T_1,T_{2},\dots, T_{k}\}$ is the simple cycle provided by Lemma~\ref{lem:simSubCyc}, with $w=T_k\cap T_1$ (possibly) being the only vertex in which two consecutive trees intersect that did not appear as one of the intersection vertices~$v_i$'s. Let us now give an estimate on the weighted girth of~$\mathscr M$. We can break down the weighted girth of $\mathscr L$ into
    \[
    k+(n-k)+\abs{\{i\in\Z_{k-1}:\lab{v_i}\ge3\}}+\abs{\{i\in\Z_{n-k+1}:\lab{v_i}\ge3\}}
    \]
    and the weighted girth of $\mathscr M$ into
    \[
    k+\abs{\{i\in\Z_{k-1}:\lab{v_i}\ge3\}}+\tau(w),
    \]
    where $\tau(w)\in\{0,1\}$, depending on whether $\lab w= 2$ or $\lab w \ge 3$. We get
    \[
    \wg{\mathscr L}-\wg{\mathscr M}=(n-k)+\abs{\{i\in\Z_{n-k+1}:\lab{v_i}\ge3\}}-\tau(w)
    \]
    with the remark that $n-k\ge1$, because $\mathscr L$ is not simple. Therefore, $\wg\L\ge\wg{\mathscr M}$ and the only way to obtain $\wg {\mathscr L}=\wg{\mathscr M}$ is to have $n-k=1$, $\{i\in\Z_{n-k+1}:\lab{v_i}\ge3\}=\varnothing$ and $\lab w\ge3$, which is not possible. Indeed, the complementary sub-cycle $\{T_k,T_n,T_1\}$ is a cycle of three standard trees, while $\girth\Gamma\ge4$, contradicting Lemma~\ref{lem:lenCycStdTrs}.
\end{proof}

\begin{proof}[Proof of Lemma~\ref{lem:wgCycStdTrs}, triangle-free case]
    The proof of the first claim is a straightforward combination of Lemma~\ref{lem:wgSimCycStdTrs} and Lemma~\ref{lem:wg simp < wg}.\medskip
    
    Let us now show the second claim, so let $\girth\Gamma\ge5$. Let us set $n=\wg\Gamma$. By the first claim together with Example~\ref{ex:canonCycStdTreesWG}, $n$ is minimal and hence $\mathscr L$ is a simple cycle, by Lemma~\ref{lem:wg simp < wg}. Let use the same construction and notations that were introduced in the proof of Lemma~\ref{lem:wgSimCycStdTrs}: the claim will follow from a more careful examination of the curvatures. In particular, let us focus on the disc diagram~$F$, with combinatorial Gauss--Bonnet theorem giving Equation~\ref{eq:cgb for E}.
    
    Let us assume by contradiction that the number of vertices of $F$ that contribute with positive curvature is strictly less than $n$. Then we have
    \[
    2\pi=\sum_{v\in\ver F}\kappa'(v)+\sum_{\overline P\in F^{(2)}}\kappa'(\overline P)<n\frac\pi2+\kappa'(\overline{P}_0)\le n\frac\pi2+\frac\pi2(4-\wg\Gamma)=2\pi,
    \]
    a contradiction. It follows that the positive curvature that vertices of $F$ contribute with is precisely $n\frac\pi2$. \par 
    Let us assume by contradiction that $F$ consists of at least two distinct $2$-cells, say $\overline{P}_0$ (namely, the one on which the redistribution of curvature was performed) and $\overline{P}_1$. Then
    \begin{align*}
        2\pi&=\sum_{v\in\ver F}\kappa'(v)+\sum_{\overline P\in F^{(2)}}\kappa'(\overline P)     &\text{(Equation~\ref{eq:cgb for E})}\\
            &\le n\frac\pi2+\kappa'(\overline{P}_0)+\kappa'(\overline{P}_1)                     &\text{(non-positive curvatures)}\\
            &\le n\frac\pi2+\frac\pi2(4-\wg{\Gamma})+\frac\pi2(4-\girth{\Gamma})                &\text{(definition of $\kappa'(\overline P_0)$ and Claim~\ref{claim:FcellNonPosCurv})}\\
            &=4\pi-\girth\Gamma\frac\pi2                                                        &\text{($n=\wg\Gamma$)}\\
            &=\frac32\pi                                                                        &\text{($\girth\Gamma\ge5$)}
    \end{align*}
    which is again a contradiction. Hence, it follows that $\overline{P}_0$ is the unique $2$-cell of $F$. Equation~\ref{eq:cgb for E} becomes
    \[
    2\pi=n\frac\pi2+\kappa'(\overline{P}_0)\le n\frac\pi2+\frac\pi2(4-\wg{\Gamma}),
    \]
    showing that $\overline P_0$ has $\kappa'$-curvature precisely $\frac\pi2(4-\wg{\Gamma})$. Therefore, the trees $T_i$'s bound a simple loop of weighted girth $n$ in $\operatorname{Link}_{C_\Gamma}(u_0)\cong\Gamma$.
\end{proof}

\subsubsection{Hyperbolic-type case} We now deal with two-dimensional Artin groups of hyperbolic type, dropping the triangle-free hypothesis.

\begin{lem}
    \label{lem:no C4 in (3,5)}
    Let $A_\Gamma$ be a two-dimensional Artin of hyperbolic type group with $\girth\Gamma=3$. For every cycle of standard trees $\L\subseteq D_\Gamma$, $\wg\L\ge5$. 
\end{lem}

\begin{proof}
    Let us assume by contradiction that there exists a cycle of standard trees $\L$ of weighted girth~$3$ or~$4$. If $\wg\L=3$, then the only possibility is that $\L$ consists of a cycle of three standard trees where all three intersection vertices $v_i$ have label~$2$, thus curvature $\kappa(v_i)\le\frac\pi2$. All other vertices and triangles have non-positive curvature by Lemma~\ref{lem:cpt curvatures}. The combinatorial Gauss--Bonnet theorem applied to any filling diagram~$D$ for $\L$ gives
    \[
    2\pi=\sum_{v\in D^{(0)}}\kappa(v)+\sum_{P\in D^{(2)}}\kappa(P)\le\sum_{i=1}^3\kappa(v_i)\le\frac32\pi,
    \]
    a contradiction.\medskip
    
    Let us now assume that $\wg\L=4$. The only possible configurations are the two depicted in Figure~\ref{fig:std trees of C4}. 
    
\begin{figure}[htp]
    \centering
    \subfloat[][]{\includegraphics[width=0.2\linewidth]{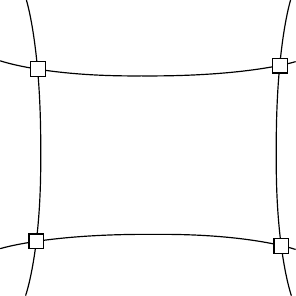}\label{fig:subfig AA}}\quad\subfloat[][]{\includegraphics[width=0.2\linewidth]{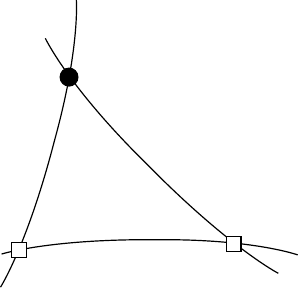}\label{fig:subfig BB}}
    \caption{The two possible configurations of a cycle of standard trees of weighted girth~$4$. White squares denote vertices with label~$2$; black circles denote vertices with label at least $3$.}
    \label{fig:std trees of C4}
\end{figure}

Let us first assume that $\L$ is a square of standard trees with four right angles (Figure~\ref{fig:subfig AA}) and let $D$ be a filling diagram for $\L$. By a standard Gauss--Bonnet argument we obtain
    \begin{align*}
        2\pi&=\sum_{v\in\ver{\partial D}}\kappa(v)+\sum_{u\in\ver{D^{\circ}}}\kappa(u)+\sum_{P\in D^{(2)}}\kappa(P) &\text{(combinatorial Gauss--Bonnet)}\\
        &\le\sum_{v\in\ver{\partial D}}\kappa(v)+\sum_{u\in\ver{D^{\circ}}}\kappa(u) &\text{(triangles have $\le0$ curvature)}\\
         &\le\sum_{v\in\ver{\partial D}}\kappa(v)+\kappa(v_0)+\sum_{u\in\ver{D^{\circ}}\setminus\{v_0\}}\kappa(u) &\text{(for some vertex $v_0$ of type~$0$)}\\
        &\le\sum_{v\in\ver{\partial D}}\kappa(v)+\kappa(v_0) &\text{(interior vertices have $\le0$ curvature)}\\
        &\le4\cdot \frac\pi2 +\kappa(v_0)   &\text{(there are four standard trees)}\\
        &<2\pi&\text{($\kappa(v_0)<0$ by Lemma~\ref{lem:cpt curvatures})}
    \end{align*}
    which is again a contradiction.\par

    The case where $\L$ is a triangle of standard trees with two right angles (Figure~\ref{fig:subfig BB}) is analogous. In particular, the two right angles can contribute with curvature at most $\frac\pi2$ each, while the third vertex can contribute with curvature at most~$\pi$.
\end{proof}

\begin{lem}
    \label{lem:no C5 in (3,6)}
    Let $A_\Gamma$ be a two-dimensional Artin group of hyperbolic type with $\girth\Gamma=3$ and $\wg\Gamma=6$. For every cycle of standard trees $\L\subseteq D_\Gamma$, $\wg\L\ge6$.
\end{lem}

\begin{proof}
    Because $\wg\Gamma=6$ and $A_\Gamma$ is of hyperbolic type, $\Gamma$ does not contain any $(2,2,2,2,2)$-pentagons nor $(p,2,2,2)$-squares nor $(p,q,2)$-triangles (for $p,q\in\N_{\ge3}$). Let us assume by contradiction that there exists a cycle of standard trees $\L\subseteq D_\Gamma$ with $\wg\L\le5$. The argument of Lemma~\ref{lem:no C4 in (3,5)} excludes the cases $\wg\L\in\{3,4\}$, so we assume that $\wg\L=5$. The possible configurations of standard trees for $\L$ are those of Figure~\ref{fig:configurations of standard trees}.

    \begin{figure}[htp]
    \centering
    \subfloat[][]{\includegraphics[width=0.2\linewidth]{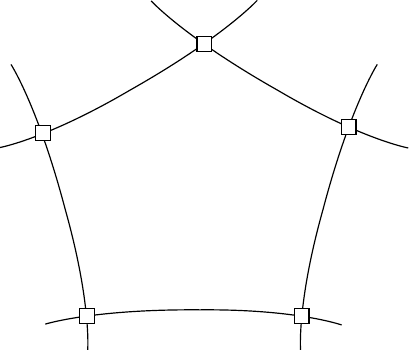}\label{fig:subfig A}}\quad\subfloat[][]{\includegraphics[width=0.2\linewidth]{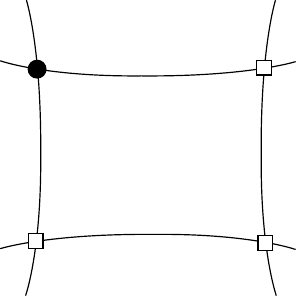}\label{fig:subfig B}}\quad\subfloat[][]{\includegraphics[width=0.2\linewidth]{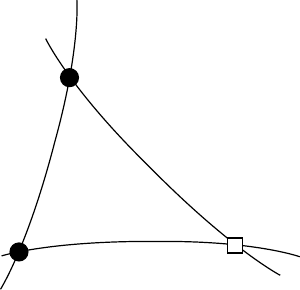}\label{fig:subfig C}}\\
    \subfloat[][]{\includegraphics[width=0.35\linewidth]{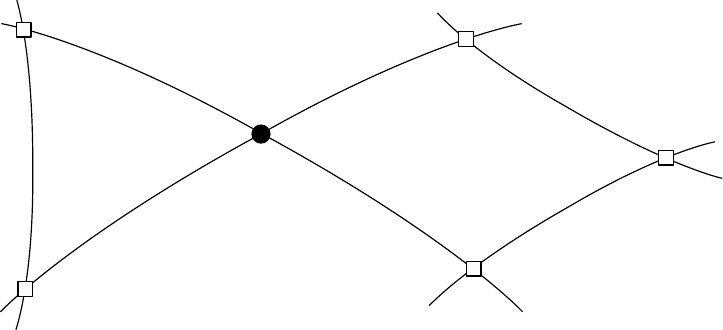}\label{fig:subfig D}}\quad \subfloat[][]{\includegraphics[width=0.25\linewidth]{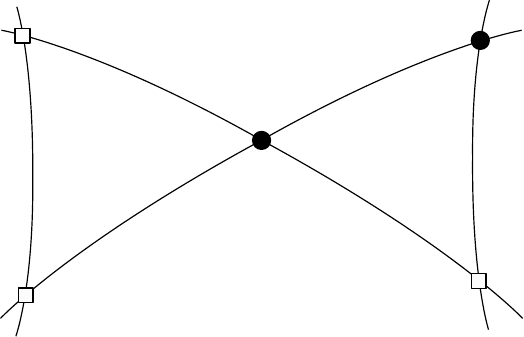}\label{fig:subfig E}}
    \caption{The five possible configurations of cycles of standard trees of weighted girth~$5$. White squares denote vertices with label~$2$; black circles denote vertices with label at least $3$.}
    \label{fig:configurations of standard trees}
\end{figure}

    For a simple cycle of standard trees, we fill in a disc diagram $D$ and use the combinatorial Gauss--Bonnet theorem to derive a contradiction, as described in Section~\ref{sec:comb gbt}. By Lemma~\ref{lem:cpt curvatures}, $2$-cells of $D$ have non-positive curvature. Moreover, vertices of type~$1$ or type~$2$ have non-positive curvature, except for the intersection vertices $\{v_i\}_{i=1}^k$. Therefore, the combinatorial Gauss--Bonnet theorem gives a first estimate of the form
    \begin{equation}
        \label{eqn:Gauss--Bonnet}
        2\pi\le \sum_{i=1}^k\kappa(v_i)+\sum_{v\in\ver D}\kappa(v)\le\sum_{i=1}^k\kappa(v_i)+\sum_{v\in\operatorname{V}^0(D)}\kappa(v),
    \end{equation}
    where $\operatorname{V}^0(D)$ denotes the vertices of $D$ of type~$0$.  Let $v$ be such a vertex and let $\gamma\subseteq\Gamma$ be a simple loop such that
    \[
    \kappa(v)\le(2-\girth\gamma)\pi-\pi\sum_{e\in\edge\gamma}\frac 1 {m_e}
    \]
    (such a vertex of type~$0$ exists by Remark~\ref{rem:verTyp0}). We obtain the following estimates:
    \begin{itemize}
        \item if $\girth\gamma=3$, then let $\alpha\coloneqq 1/p+1/q+1/r$, where $p,q,r\in\N_{\ge3}$ are the labels of the triangle~$\gamma$. Since $A_\Gamma$ is of hyperbolic type, $\alpha<1$ and the above equation gives $\kappa(v)\le(\alpha-1)\pi$;
        \item if $\girth\gamma=4$, then 
        \[
        \kappa (v)\le-2\pi+\pi\sum_{e\in\edge\gamma}\frac 1 {m_e}\le -2\pi+\pi\left(2\cdot \frac 1 2+2\cdot\frac 1 3\right)=-\frac \pi 3,
        \]
        where we use that $\Gamma$ does not contain any $(p,2,2,2)$-squares for $p\in\N_{\ge2}$;
        \item if $\girth\gamma=5$, then
        \[
        \kappa (v)\le-3\pi+\pi\sum_{e\in\edge\gamma}\frac 1 {m_e}\le -3\pi+\pi\left(4\cdot \frac 1 2+\frac 1 3\right)=-\frac 2 3 \pi,
        \]
        where we use that $\Gamma$ does not contain any $(2,2,2,2,2)$-pentagons;
        \item if $\girth\gamma\ge6$, then 
        \[
        \kappa(\gamma)\le(2-\girth\gamma)\pi+\frac \pi 2\girth\gamma\le-\pi.
        \]
    \end{itemize}
    The proof now reduces to checking that none of the cases of Figure~\ref{fig:configurations of standard trees} is possible. Let us partition the set of vertices of type~$0$ of $D$: for $i\in\{3,4,5\}$, let $V_i$ be the set of vertices~$u$ with $\girth{\link D u}=i$ and let $V$ be the set of vertices~$u$ with $\girth{\link D u}\ge6$. \medskip
    
    If $\partial{D}$ consists of a right-angled pentagon (Figure~\ref{fig:subfig A}), then Equation~\ref{eqn:Gauss--Bonnet} yields $2\pi\le5\cdot\frac\pi 2+\sum_v\kappa(v)$ or, equivalently, $0\le\frac\pi2+\sum_{v}\kappa(v)$. If $V\ne\varnothing$ or $V_5\ne\varnothing$, then we readily obtain a contradiction. Therefore, let us assume that $V,V_5=\varnothing$. The standard trees of the pentagon intersect in vertices with label $2$, while triangles of $\Gamma$ do not contain edges labelled by $2$. It follows that every intersection vertex is adjacent to a vertex in $V_4$. If $\abs{V_4}\ge2$, then $V_4$ would contribute with a curvature at most $-\frac23\pi$, giving a contradiction. Therefore $V_4$ consists of a single vertex $u$ and $D=\operatorname{Star}_D(u)$. It follows that there exists a loop $\gamma\subseteq\Gamma$ of girth at most $4$ such that the image of $D$ in $D_\Gamma$ is contained in a translate of $K_\gamma$, a contradiction. \medskip

    If $\partial D$ is a square with three right-angles (Figure~\ref{fig:subfig B}), then Equation~\ref{eqn:Gauss--Bonnet} gives $2\pi\le3\cdot\frac\pi2+\pi+\sum_v\kappa(v)$ or, equivalently, $0\le\frac\pi2+\sum_v\kappa(v)$ and an argument similar to the one of the previous paragraph applies. If $V$ or $V_5$ is non-empty, then vertices of type~$0$ would contribute with an excess of negative curvature, a contradiction. Because standard trees intersect in three vertices of label~$2$ and triangles of $\Gamma$ have labels at least~$3$, it follows that every such intersection vertex is adjacent to a vertex of $V_4$. Again, if $\abs{V_4}\ge2$, then there would be an excess of negative curvature, against the combinatorial Gauss--Bonnet theorem. Therefore the three intersection vertices labelled by~$2$ belong to the link of a unique vertex of $V_4$. Again, this yields a contradiction, since squares of $\Gamma$ can have at most two edges with label~$2$.  \medskip

    If $\partial D$ is a triangle with one right angle (Figure~\ref{fig:subfig C}), then Equation~\ref{eqn:Gauss--Bonnet} gives again $0\le\frac\pi2+\sum_v\kappa(v)$. Arguing as above, there must be exactly one vertex $u\in V_4$. Let $\lambda\subseteq\Gamma$ be a square subgraph such that $\kappa(u)\le\pi(2-\girth\lambda)+\pi\sum_{e\in\edge\lambda}\frac{1}{m_e}$, as provided by Lemma~\ref{lem:cpt curvatures}. We shall distinguish two cases, depending on the labels of~$\lambda$. \par 
    If $\lambda$ has two edges labelled by~$2$, then $\link D u$ contains a vertex of type~$2$ and label~$2$ and it follows that there must be at least another vertex of $V_4$ by~\cite{appel1983artin}*{Lemma~6}, giving an excess of negative curvature, a contradiction. \par 
    If $\lambda$ has exactly one edge labelled by~$2$, then we can have a sharper bound for the curvature with
    \[
    \kappa(u)\le\pi(2-4)+\sum_{e\in\edge\lambda}\frac{1}{m_e}\le-2\pi+\frac\pi2+3\cdot\frac\pi3=-\frac\pi2
    \]
    and Equation~\ref{eqn:Gauss--Bonnet} becomes $0\le\sum_{v\ne u}\kappa(v)$. If $D$ contains another vertex of type~$0$, then there would be an excess of negative curvature, giving a contradiction. It follows that $D$ coincides with the star of $u$ in $D$, against the hypothesis that $\link D u$ has girth~$4$. \medskip

    In the cases of Figure~\ref{fig:subfig D} and Figure~\ref{fig:subfig E}, we have four simple sub-cycle of standard trees to consider. The only case that was not previously covered is the one of a triangle of standard trees with exactly two right angles. The combinatorial Gauss--Bonnet theorem gives
    \[
    2\pi\le2\cdot\frac\pi2+\pi+\sum_{v\in\ver D}\kappa(v),
    \]
    which easily shows a contradiction, since there exists at least one vertex of type~$0$, which has strictly negative curvature.\medskip 
    
    This exhausts all the possible configurations of standard trees of weighted girth~$5$ and concludes the proof. 
\end{proof}

\begin{proof}[Proof of Lemma~\ref{lem:wgCycStdTrs}, hyperbolic-type case]
    It is not restrictive to assume that $\Gamma$ contains an induced triangle, for otherwise the \emph{triangle-free case} applies. Because $\Gamma$ is two dimensional of hyperbolic type and contains an induced triangle, $\wg\Gamma\in\{5,6\}$. The combination of Lemma~\ref{lem:no C4 in (3,5)} and Lemma~\ref{lem:no C5 in (3,6)} yields the claim.
\end{proof}

\section{Minimal embeddings of cycle Artin groups}
\label{sec:minCycEmbedd}

For all Artin groups, we give an algebraic characterisation of the girth of their defining graph.

\begin{thm}
    \label{thm:girIsoInv}
    Let $\Gamma$ be a labelled graph; $\girth\Gamma$ coincides with the minimum $n\in\N_{\ge3}$ such that there exists a labelled cycle $C_n$ on $n$ vertices and an embedding $A_{C_n}\to A_\Gamma$.
\end{thm}

Moreover, by restricting to two-dimensional Artin groups of hyperbolic type, we will show a similar algebraic characterisation of the weighted girth.

\begin{thm}
    \label{thm:wgIsoInv}
    Let $A_\Gamma$ be a two-dimensional Artin group of hyperbolic type. The weighted girth $\wg\Gamma$ of~$\Gamma$ coincides with the minimum $n\in\N_{\ge3}$ such that there exists a cycle $C_n$ on $n$ vertices and an embedding $R_{C_n}\to A_\Gamma$.
\end{thm}

The section will be devoted to proving Theorem~\ref{thm:girIsoInv} and Theorem~\ref{thm:wgIsoInv}. We start with the elementary observation that cycle subgraphs of the defining graph induce embeddings of a cycle Artin group (resp. of a cycle RAAG) of the corresponding girth (resp. weighted girth):

\begin{example}[canonical cycle embeddings]
    \label{ex:cannEmbedd}
    Let $\Gamma$ be a labelled graph. For every cycle subgraph $\gamma\subseteq\Gamma$, the cycle Artin group $A_\gamma$ embeds into $A_\Gamma$ as a standard parabolic subgroup.\par
    Moreover, let us assume that $\Gamma$ is two dimensional. Denote by $\{a_i\}_{i}$ the set of vertices of $\gamma$ and, if $m_{a_i,a_{i+1}}\ge3$, then denote by $z_i$ the generator of the centre of $\Span{a_i,a_{i+1}}$. The subgroup of $A_\Gamma$ generated by the $g_i^4$'s and $z_j^4$'s is isomorphic to the right-angled Artin group based on the cycle $C_{\wg\gamma}$~\cite{jankiewicz2022right}*{Theorem 1.1}.
\end{example}

\subsection{Characterisation of girth} A cycle-of-standard-trees argument provides the claim in the case of dimension~$2$ and hyperbolic type.

\begin{lem}
    \label{lem:girIsoInv}
    Let $A_\Gamma$ be a two-dimensional Artin group of hyperbolic type; $\girth\Gamma$ coincides with the minimum $n\in\N_{\ge3}$ such that there exists a labelled cycle $C_n$ on $n$ vertices and an embedding $A_{C_n}\to A_\Gamma$.
\end{lem}

\begin{proof}
    We can assume that $\Gamma$ has finite girth, for otherwise the claim trivially holds. Indeed, if $\Gamma$ is a tree, then it is two dimensional of hyperbolic type (the characterisation of Definition~\ref{def:hypTyp} checks a vacuous condition) and therefore any cycle Artin group embedding induces a cycle of standard trees in $C_\Gamma$ (see Lemma~\ref{lem:mono->cycle}). However, when $\Gamma$ is a tree, $C_\Gamma$ contains no cycles of standard trees, as observed in Remark~\ref{rem:treeDefGrph}.\par 
    Example~\ref{ex:cannEmbedd} provides a labelled cycle $C_{\girth\Gamma}$ and an embedding $A_{C_{\girth\Gamma}}\to A_\Gamma$.  To show that $\girth\Gamma$ is minimal, let $n\in\N_{\ge3}$, let $C_n$ be a labelled cycle and assume there is a group monomorphism $\phi\colon A_{C_n}\to A_\Gamma$. By Lemma~\ref{lem:cycle of std trees}, there are $k\in\{3,\dots,n\}$ and a cycle of standard trees $\{T_i\}_{i\in\Z_k}$ associated with $\phi$. By Lemma~\ref{lem:lenCycStdTrs}, $\girth\Gamma\le k$.
\end{proof}

An ad-hoc argument covers the remaining case, i.e. when $A_\Gamma$ is not of hyperbolic type.

\begin{lem}
    \label{lem:gir3VsGir4}
    Let $\Delta$ be a labelled triangle and let $\Gamma$ be a triangle-free labelled graph. There exists no group monomorphism $A_\Delta\to A_\Gamma$.
\end{lem}

\begin{proof}
    Let us assume by contradiction that there exists a group monomorphism $\phi\colon A_\Delta\to A_\Gamma$. Because $\girth\Gamma\ge4$, $A_\Gamma$ is two-dimensional, and so is $A_\Delta$, for otherwise $\phi(A_\Delta)$ would contain a subgroup isomorphic to~$\Z^3$. Let $\ver\Delta=\{a,b,c\}$: because $A_\Delta$ is two dimensional, at most one between $m_{ab}$, $m_{bc}$ and $m_{ac}$ can be a~$2$, say $m_{ac}$.  Because $\girth\Gamma\ge4$ (in fact, because $\Gamma$ has no \emph{euclidean triangles}), each of the subgroups $\phi(\Span {a,b})$ and $\phi(\Span{b,c})$ is conjugated into a dihedral parabolic subgroup of~$A_\Gamma$~\cite{vaskou2022isomorphism}*{Corollary 3.12}. It follows that $\phi(a)$, $\phi(b)$ and $\phi(c)$ act elliptically on $C_\Gamma$ and the lemmas of Section~\ref{subsec:cycStdTrsFromCycArt} verbatim extend to this case. In particular, $\phi(a)$, $\phi(b)$ and $\phi(c)$ act tree-elliptically on $C_\Gamma$ and their fixed-point sets form a cycle of three standard trees, contradicting Lemma~\ref{lem:lenCycStdTrs}.
\end{proof}
    
\begin{rem}
    The proof of Lemma~\ref{lem:gir3VsGir4} relies on a particular case of Vaskou's characterisation of dihedral Artin subgroups of two-dimensional Artin groups~\cite{vaskou2022isomorphism}*{Theorem D}. More precisely, we use that all subgroups of a triangle-free Artin group that are abstractly isomorphic to a dihedral Artin group are conjugated into a dihedral \emph{parabolic} subgroup. Remarkably, the proof of Lemma~\ref{lem:girIsoInv} does not rely on this classification. 
\end{rem}

\begin{proof}[Proof of Theorem~\ref{thm:girIsoInv}]
    Again Example~\ref{ex:cannEmbedd} provides a labelled cycle $C_{\girth\Gamma}$ and an embedding $A_{C_{\girth\Gamma}}\to A_\Gamma$. If $\girth\Gamma=3$, then it is minimal by definition and there is nothing to prove. If $\girth\Gamma=4$, then Lemma~\ref{lem:gir3VsGir4} rules out the possibility of embedding any $A_{C_3}$. If $\girth\Gamma\ge5$, then $A_\Gamma$ is two dimensional of hyperbolic type and Lemma~\ref{lem:girIsoInv} applies.
\end{proof}

\subsection{Characterisation of weighted girth} Let us now turn to the weighted girth case. Recall that we restrict to two-dimensional and hyperbolic-type Artin groups.

\begin{proof}[Proof of Theorem~\ref{thm:wgIsoInv}]
    As observed in the proof of Theorem~\ref{thm:girIsoInv}, we can assume that $\Gamma$ is not a tree. Example~\ref{ex:cannEmbedd} provides a labelled cycle $C_{\wg\Gamma}$ and an embedding $R_{C_{\girth\Gamma}}\to A_\Gamma$. In order to show that this embedding is minimal, let $n\in\N_{\ge3}$, let $C_n$ be a cycle graph and assume there is a group monomorphism $\phi\colon R_{C_n}\to A_\Gamma$. Let $\L$ be the cycle of standard trees associated with~$\phi$. By Lemma~\ref{lem:cycStdTrsAssToRaag}, $n\ge \wg{\mathscr L}$; by Lemma~\ref{lem:wgCycStdTrs}, $\wg\L\ge\wg\Gamma$.
\end{proof}

\section{Girth of the commutation graph}
\label{sec:girthCrvGrph}

The commutation graph of an Artin group was introduced by Hagen-Martin-Sisto and, in the two-dimensional hyperbolic-type case, it plays the role of the curve curve graph for mapping class groups~\cite{hagen2024extra}. When $\Gamma$ is furthermore without leaves, it coincides with the intersection graph introduced by Huang-Osajda-Vaskou~\cite{huang2024rigidity}. \medskip 

Let $\Gamma$ be a labelled graph; two vertices $a$ and $b$ are conjugate to each other in $A_\Gamma$ if and only if there is an odd-labelled path in $\Gamma$ between $a$ and $b$~\cite{paris1997parabolic}*{Corollary 4.2}. Let $\operatorname{V}_{\textup{odd}}(\Gamma)$ denote the quotient set of $\ver\Gamma$ under the equivalence relation of conjugation. 
Throughout the section we assume that $A_\Gamma$ is a two-dimensional Artin group of hyperbolic type.

\begin{defi}[commutation graph]
    Let us define the set 
    \[
    \mathcal H_\Gamma = \{\norm{A_\Gamma}{a}:a\in\operatorname{V}_{\textup{odd}}(\Gamma)\}\cup\{\Span{b,c}:\{b,c\}\in\edge\Gamma\text{ and }m_{bc}\ge3\};
    \]
    the \emph{(algebraic) commutation graph} of $\Gamma$ is the graph $Y_\Gamma$ with vertices
    \[
    \bigsqcup_{H\in\mathcal H_\Gamma}A_\Gamma/H
    \]
    and two vertices $g{G}$ and $h{H}$ are adjacent if $\conj {\operatorname{Z}(G)} g$ and $\conj {\operatorname{Z}(H)} h$ elementwise commute.  
\end{defi}

The group $A_\Gamma$ acts on the vertices of $Y_\Gamma$ by left multiplication, which induces a simplicial action on $Y_\Gamma$. Note that the adjacency condition justifies the name \emph{commutation graph}. \par
Recall that a \emph{leaf} of a graph is a vertex of degree one. We are going to show the following equality.

\begin{thm}
    \label{thm:girthCrvGrph}
    If $\Gamma$ has no leaves, then $\girth{Y_\Gamma}=\wg\Gamma$.
\end{thm}

\begin{rem}[algebraic characterisation of~$Y_\Gamma$]
    Under the hypothesis of Theorem~\ref{thm:girthCrvGrph}, i.e. when $A_\Gamma$ is a two-dimensional Artin group of hyperbolic type and $\Gamma$ has no leaves, the algebraic commutation graph~$Y_\Gamma$ is algebraically characterised, hence an isomorphism invariant. Indeed, the subgroups
    \[
    \{\norm{A_\Gamma}a:a\in\ver\Gamma\}\cup\{\Span{b,c}:\{b,c\}\in\edge\Gamma\text{ with }m_{bc}\ge3\}
    \]
    coincide, up to conjugation, with the inclusion-wise maximal subgroups of $A_\Gamma$ that virtually split as a non-trivial direct product, as a consequence of the fact that $A_\Gamma$ is acylindrically hyperbolic~\cite{martin2022acylindrical}*{Theorem A}.
\end{rem}

For our purposes, we define a different graph, the \emph{geometric} commutation graph, and then show that the latter is equivariantly isomorphic to the algebraic commutation graph, when the hypothesis of Theorem~\ref{thm:girthCrvGrph} is satisfied (i.e. when $\Gamma$ has no leaves). 

\begin{defi}[geometric commutation graph]
    The \emph{geometric commutation graph} of $\Gamma$ is the simplicial graph $\curve\Gamma$ whose vertices come in two types
    \[
    \{g\cdot T_a: a\in\ver\Gamma,g\in A_\Gamma\}\cup\{h\cdot\Span{b,c}:\{b,c\}\in\edge\Gamma, m_{bc}\ge3 \text{ and }h\in A_\Gamma\}
    \]
    (every vertex is either a standard tree in $D_\Gamma$ or a type-$2$ vertex of label at least~$3$). Assume that the distinct standard trees $g\cdot T_a$ and $h\cdot T_b$ intersect. Then they intersect at a single type-$2$ vertex $f\cdot\Span{c,d}$:
    \begin{itemize}
        \item if $m_{cd}=2$, then there is an edge in $\curve\Gamma$ between $g\cdot T_a$ and $h\cdot T_b$ (note that $f\cdot \Span{c,d}$ is not a vertex of $\curve\Gamma$);
        \item if $m_{cd}\ge3$, then $f\cdot \Span{c,d}$ is adjacent in $\curve \Gamma$ to $g \cdot T_a$ and $h\cdot T_b$.
    \end{itemize}
\end{defi}

The group $A_\Gamma$ acts on $\curve\Gamma$ by left multiplication and the action preserves the partition of the vertices. \par

We introduce the following nomenclature of vertices of $\curve\Gamma$:
\begin{itemize}
    \item \emph{tree-type} vertices are those of the form $g \cdot T_a$;
    \item \emph{dihedral-type} vertices are those of the form $h\cdot\Span{b,c}$.
\end{itemize}

Let $\cst\Gamma$ denote the set of cycles of standard trees in $D_\Gamma$ and let $\sl\Gamma$ denote the set of simple loop in $\curve\Gamma$.

\begin{lem}
    \label{lem:girGeomCommGrph}
     There is a bijection $f\colon\cst\Gamma\to\sl\Gamma$ such that, for every cycle of standard trees $\mathscr L\in\cst\Gamma$, $\girth{f(\mathscr L)}=\wg{\mathscr L}$. \par
     In particular, $\wg\Gamma=\girth{\curve\Gamma}$.
\end{lem}

\begin{proof}
    The proof relies on the following observation:

    \begin{claim}
        \label{claim:loopsInCurve}
        Every  loop in $\curve\Gamma$ is uniquely determined by the (ordered) sequence of tree-type vertices it traverses.
    \end{claim}

    \begin{proof}[Proof of Claim~\ref{claim:loopsInCurve}]
        For a loop subgraph of $\curve\Gamma$, let $\{u_i\}_{i\in\Z_k}$ be the sequence of its tree-type vertices (up to the choice of an initial vertex). If $u_i$ and $u_{i+1}$ are at combinatorial distance $1$, then there is nothing to show. If they are at combinatorial distance~$2$, it means that the two standard trees corresponding to $u_i$ and $u_{i+1}$ intersect. Because distinct standard trees intersect in at most one vertex~\cite{martin2023tits}*{Remark 4.4}, the dihedral-type vertex adjacent to $u_i$ and $u_{i+1}$ is uniquely determined. Finally, $u_i$ and $u_{i+1}$ cannot be at combinatorial distance at least~$3$, for no two dihedral-type vertices can be adjacent in $\curve \Gamma$.
    \end{proof}

    For every cycle of standard trees $\L\subseteq D_\Gamma$, if $\L=\{T_i\}_{i\in\Z_k}$, then we define $f(\L)$ to be the unique loop in $\curve\Gamma$ determined by the tree-type vertices $\{T_i\}_{i\in\Z_k}$. By Claim~\ref{claim:loopsInCurve}, the definition of $f$ is well posed. Furthermore, because a cycle of standard trees consists of distinct trees, the image of $f$ is contained in the subset of simple loops in $\curve\Gamma$. Also, for every $\L\in\cst\Gamma$, $\wg\L=\girth{f(\L)}$ as the dihedral-type vertices of $\curve\Gamma$ correspond to type-$2$ vertices of $D_\Gamma$ with label at least~$3$.\par
    Let us define an inverse $g\colon\sl\Gamma\to\cst\Gamma$ for $f$ in the following way. For a simple loop $\gamma\in\sl\Gamma$, let $\{T_i\}_{i\in\Z_k}$ be the sequence of its tree-type vertices: we define $g(\gamma)=\{T_i\}_{i\in\Z_k}$. Let us check that $g(\gamma)$ is indeed a cycle of standard trees (recall Definition~\ref{def:cycStdTree}). Because $\gamma$ is simple, the $T_i$'s are all distinct. Two consecutive tree-type vertices $T_i$ and $T_{i+1}$ are at combinatorial distance at most~$2$: if they are adjacent, then the standard trees $T_i$ and $T_{i+1}$ intersect in a type-$2$ vertex of label~$2$; if they are at combinatorial distance~$2$, then the standard trees $T_i$ and $T_{i+1}$ intersect in a type-$2$ vertex of label at least~$3$. In particular, $T_{i}\cap T_{i+1}$ is non-empty. Finally, to show that $T_{i-1}\cap T_i\cap T_{i+1}=\varnothing$, notice that, if the three distinct standard trees intersected in a single vertex, then such vertex would correspond to a backtracking dihedral-type point in $\gamma$, against the assumption that $\gamma$ is simple. This show that $g$ is well defined. By construction, $f$ and $g$ are each other's inverses.
\end{proof}

Note that Lemma~\ref{lem:girGeomCommGrph} holds without assuming that $\Gamma$ has no leaves.

\begin{lem}
    \label{lem:isomAlgGeoCommGrph}
    Assume that $\Gamma$ has no leaves. There is an $A_\Gamma$-equivariant graph isomorphism between $Y_\Gamma$ and $\curve\Gamma$.
\end{lem}

\begin{proof}
    We first define an intermediate graph $X_\Gamma$ in the way that follows and see that $Y_\Gamma$ and $\curve\Gamma$ are both $A_\Gamma$-equivariantly isomorphic to~$X_\Gamma$. Vertices of $X_\Gamma$ are $A_\Gamma$-conjugates of
    \[
    \{\norm{A_\Gamma}{a}:a\in\ver\Gamma\}\cup\{\Span{b,c}:\{b,c\}\in\edge\Gamma\text{ and }m_{bc}\ge3 \},
    \]
    where the vertices $\conj G g$ and $\conj H h$ are adjacent if the corresponding centres elementwise commute. Then $X_\Gamma$ is an $A_\Gamma$-graph when endowed with the action of $A_\Gamma$ on $\ver{X_\Gamma}$ by conjugation.

    \begin{claim}\label{claim:XisomY}
        The graphs $X_\Gamma$ and $Y_\Gamma$ are $A_\Gamma$-equivariantly isomorphic.
    \end{claim}

    \begin{proof}[Proof of Claim~\ref{claim:XisomY}]
        Let $f\colon Y_\Gamma\to X_\Gamma$ be defined as follows: for every $gG\in\ver{Y_\Gamma}$, $f(gG)=\conj G g$. By construction, $f$ is an $A_\Gamma$-equivariant graph homomorphism. It is also straightforward to see that $f$ is surjective. \par
        Let us show that $f$ is injective. Assume by contradiction that there are cosets $gG$ and $hH$ such that $\conj G g=\conj H h$. Because $A_\Gamma$ is two-dimensional, normalisers of standard generators are of the form $\Z\times F$ (for some finite-rank free group $F$~\cite{martin2022acylindrical}), while non-abelian dihedral Artin groups are not RAAGs, it follows that $G$ and $H$ are either both the normaliser of a standard generator or both dihedral parabolic subgroups. Distinct spherical-type dihedral standard parabolic subgroups are never conjugate to each other, while every spherical-type non-abelian dihedral parabolic subgroup is self-normalising~\cite{godelle2007artin}*{Corollary 4.12}. Therefore we can assume that $G=\norm{A_\Gamma}a$ and $H=\norm{A_\Gamma}b$ for some $a,b\in\ver\Gamma$. If $\norm{A_\Gamma}a$ and $\norm{A_\Gamma}b$ are conjugate, then so are their centres and, by means on the height homomorphism $A_\Gamma\to\Z$ that maps every standard generator to~$1$, one shows that $a$ and $b$ must be conjugate. It follows that there is an odd-labelled path between $a$ and $b$ in $\Gamma$, meaning that $a=b$ in $\operatorname{V}_{\textup{odd}}(\Gamma)$. It follows that $gh^{-1}$ normalises $\norm{A_\Gamma}a$, which is self-normalising, whence follows that $g\norm{A_\Gamma}a=h\norm{A_\Gamma}b$ (to see that $\norm{A_\Gamma}a$ is self-normalising, observe that the latter splits as a direct product $\Span a\times F$, where $\Span a$ is central, hence characteristic~\cite{hagen2024extra}*{Lemma 2.27}).
    \end{proof}

Note that the the isomorphism between $X_\Gamma$ and $Y_\Gamma$ holds in general, without the hypothesis of leafless defining graph. Let us now show that $\curve\Gamma$ and $X_\Gamma$ are $A_\Gamma$-equivariantly isomorphic. 

    \begin{claim}\label{claim:XisomCurve}
        Let us assume that $\Gamma$ has no leaves. The graphs $X_\Gamma$ and $\curve\Gamma$ are $A_\Gamma$-equivariantly isomorphic.
    \end{claim}

    \begin{proof}[Proof of Claim~\ref{claim:XisomCurve}]
        Let us define $\phi\colon\curve\Gamma\to X_\Gamma$ as follows: for every $x\in\ver{\curve\Gamma}$, $\phi(x)=\stab{A_\Gamma}{x}$. In particular, for tree-type vertices we have $\phi(g\cdot T_a) =\conj{\norm{A_\Gamma}{a}}{g}$~\cite{hagen2024extra}*{Lemma 2.27}, while for dihedral-type vertices we have $\phi(g\cdot\Span{a,b})=\conj{\Span{a,b}}g$. By construction, $\phi$ is an $A_\Gamma$-equivariant graph homomorphism and it is straightforward to check that it is surjective. \par 
        In order to show that $\phi$ is injective, we construct a left inverse. Let us define $\psi\colon X_\Gamma\to\curve\Gamma$ as follows: for every vertex $H$ of $X_\Gamma$, let $\psi(H)=\fix{\operatorname{Z}(H)}$. Let us show that $\psi$ is well defined. Up to conjugation, $H$ is of the form $\norm{A_\Gamma}a$ or $\Span{a,b}$. \par
        If $H=\norm{A_\Gamma}a$, then it is of the form $\Span{a}\times F_k$ and, because $\Gamma$ has no leaves, $k\ge2$. It follows that the centre of $\norm{A_\Gamma}a$ coincides with $\Span a$, which acts elliptically on $D_\Gamma$, fixing the standard tree $T_a$.\par
        If $H=\Span{a,b}$, then $H$ itself acts elliptically on $D_\Gamma$, fixing the type-$2$ vertex $\Span{a,b}$.\par
        By construction, $\psi\circ\phi=\operatorname{id}_{\curve\Gamma}$, showing that $\phi$ is injective.
    \end{proof}

    Combining Claim~\ref{claim:XisomY} and Claim~\ref{claim:XisomCurve} gives the required isomorphism.
\end{proof}

\begin{proof}[Proof of Theorem~\ref{thm:girthCrvGrph}]
    By Lemma~\ref{lem:girGeomCommGrph}, $\wg\Gamma=\girth{\curve\Gamma}$; by Lemma~\ref{lem:isomAlgGeoCommGrph}, $\curve{\Gamma}$ and $Y_\Gamma$ are isomorphic graphs and hence $\girth{\curve\Gamma}=\girth{Y_\Gamma}$.
\end{proof}

\subsection{Realising cycle embeddings} We end this section with an open problem. Let $\L\subseteq D_\Gamma$ be a cycle of standard trees, which we can view as a simple loop $\ell\subseteq\curve\Gamma$ by Lemma~\ref{lem:girGeomCommGrph}. Claim~\ref{claim:loopsInCurve} shows that $\ell$ is uniquely determined by the set of its tree-type vertices, let it be $\{T_{i}\}_{i\in\Z_n}$. The pointwise stabiliser of $T_i$ is infinite cyclic generated, say generated by $g_i$.

\begin{question}
    \label{quest:classifyCycArtSubGrp}
    When is the subgroup of $A_\Gamma$ generated by $\{g_i\}_{i\in\Z_n}$ a cycle Artin group?
\end{question}

The dihedral-type vertices of $\ell$ correspond to type-two vertices $v_{j}=T_j\cap T_{j+1}$ of $D_\Gamma$ with label at least~$3$. The stabiliser of $v_j$ is a spherical irreducible dihedral Artin group, therefore its centre in infinite cyclic, say generated by~$z_j$.

\begin{question}
    \label{quest:classifyCycRaagSubGrp}
    Under which conditions does there exist $k\in\N_{\ge2}$ such that the subgroup of $A_\Gamma$ generated by the $g^k_i$'s and the $z^k_j$'s is a cycle RAAG?
\end{question}

Remarkably, results of Jankiewicz-Schreve~\cite{jankiewicz2022right} and Oh-Park~\cite{oh2025embedability} go in this direction.

\section{Rigidity of cycle Artin groups}
\label{sec:rigiCycArtGrp}

We use the algebraic characterisation of girth of Theorem~\ref{thmintro:girIsoInv} to show the rigidity of Artin groups based on a cycle.

\begin{thm}
    \label{thm:rigidCycArtGrp}
    Let $\Gamma$ and $\Lambda$ be labelled cycle graphs; if $A_\Gamma\cong A_\Lambda$, then $\Gamma\cong\Lambda$.
\end{thm}

If $A_\Gamma$ and $A_\Lambda$ are two isomorphic cycle Artin groups, then Theorem~\ref{thmintro:girIsoInv} yields that $\Gamma$ and $\Lambda$ are cycles of the same length. We shall split the proof of Theorem~\ref{thm:rigidCycArtGrp} into three cases, depending on whether such length is~$3$, $4$ or at least~$5$.\par
We will often exploit the following observation.
 
\begin{lem}
    \label{lem:bound_fund_dom}
    Let $\Gamma$ and $\Lambda$ be labelled $n$-cycle graphs, let $A_\Lambda$ be two dimensional of hyperbolic type and let $\phi\colon A_\Gamma\to A_\Lambda$ be a group isomorphism. If the cycle of standard trees associated with $\phi$ bounds a single translate of $K_\Lambda$, then $\Gamma\cong\Lambda$.
\end{lem}

\begin{proof}
    Let $\L$ be the cycle of standard trees associated with $\phi$ (see Lemma~\ref{lem:mono->cycle}). Up to a translation in $D_\Lambda$ (equivalently, conjugation by a unique element in $A_\Lambda$), we can assume that $\L$ bounds the canonical fundamental domain $K_\Lambda$. By Lemma~\ref{lem:cycle of std trees}, $\L$ consists of at most $\girth\Gamma$ standard trees, while Lemma~\ref{lem:lenCycStdTrs} gives that $\L$ consists of at least $\girth\Lambda$ standard trees. Because $\girth\Gamma=\girth\Lambda=n$, $\L$ consists of precisely $n$ standard trees. In particular, $\phi$ maps standard generators of $A_\Gamma$ to standard generators of $A_\Lambda$. Therefore $\phi$ induces an injection $\phi\colon\ver\Gamma\to\ver\Lambda$ of finite sets with same cardinality, hence a bijection. Let $\{a,b\}\in\edge\Gamma$. By Lemma~\ref{lem:intersection fps}, the standard trees $T_{\phi(a)}$ and $T_{\phi(b)}$ intersect at the vertex $\Span{\phi(a),\phi(b)}$. Because $\phi(\Span{a,b})=\Span{\phi(a),\phi(b)}$, it follows that $\phi$ preserves the labels of edges. The data of vertices (standard generators), adjacency and labels (dihedral relations) completely determine the isomorphism class of the labelled graph.
\end{proof}

\begin{lem}
    \label{lem:rigid_girth>=5}
    Let $\Gamma$ and $\Lambda$ be labelled cycle graphs and let $\phi\colon A_\Gamma\to A_\Lambda$ be a group isomorphism; if $\girth\Lambda\ge5$, then $\Gamma\cong\Lambda$.
\end{lem}

\begin{proof}
    By Theorem~\ref{thmintro:girIsoInv}, the length of $\Gamma$ and $\Lambda$ coincide. It follows that both $A_\Gamma$ and $A_\Lambda$ are two dimensional of hyperbolic type. Let $\{T_i\}_{i\in\Z_k}\subseteq D_\Lambda$ be the cycle of standard trees associated with $\phi$. Lemma~\ref{lem:lenCycStdTrs} gives that $\L$ bounds a single fundamental domain of $D_\Lambda$ and therefore $\Gamma\cong\Lambda$ by Lemma~\ref{lem:bound_fund_dom}.
\end{proof}

Let us move to the square case. We refer to Lemma~\ref{lem:cpt cube curvatures} for the estimate of curvatures in the cubical Deligne complex.

\begin{lem}[grid lemma]
    \label{lem:grid}
    Let $\Gamma$ be a labelled $4$-cycle graph. Let $\{T_i\}_{i\in\Z_4}$ be a cycle of four standard trees in $C_{\Gamma}$ and let $F\colon D\to C_\Gamma$ be a reduced disc diagram associated with it. Then each intersection vertex has curvature precisely $\frac\pi2$. \par 
    Moreover, if $D$ contains a vertex $v$ of type~$2$ that is not an intersection vertex, then $\lab v=2$. If $v$ lies in the interior of $D$, then it is adjacent to four type-$1$ vertices; if $v$ lies in the boundary of $D$, then it is adjacent to two type-$1$ vertices. 
\end{lem}

\begin{proof}
    We denote by $v_i$ the intersection vertex $T_i\cap T_{i+1}$. The first claim is a direct application of the combinatorial Gauss--Bonnet theorem~\ref{thm:comb gbt}. Indeed, each of the four $v_i$'s has curvature at most $\frac\pi2$ and it cannot be strictly smaller, since every other vertex contributes with non-positive curvature. \par
    
    For the ``moreover'' claim, let us assume by contradiction that $D$ contains a type-$2$ vertex~$v$ of label $m\ge3$ and let us estimate its curvature. We shall distinguish two cases, depending on whether $v$ belongs to the interior or the boundary of $D$. \par
    
    Let us first assume that $v$ belongs to the interior of $D$. We have that $\kappa(v)\le\kappa(f(v))=2\pi-\len{ }{\gamma}$, for some cycle subgraph $\gamma\subseteq\link{C_\Gamma}{f(v)}$. Because every embedded loop of $\link{C_\Gamma}{f(v)}$ has length at least~$2m\cdot\frac\pi2=m\pi$~\cite{appel1983artin}*{Lemma~6}, we have that $\kappa(v)\le(2-m)\pi$. By the combinatorial Gauss--Bonnet theorem~\ref{thm:comb gbt} relative to the disc diagram $D$ we obtain
    \[
    2\pi=\sum_{u\in D^{(0)}}\kappa(u)+\sum_{P\in D^{(2)}}\kappa(P)\le\sum_{i=1}^4\kappa(v_i)+\kappa(v),
    \]
    where the estimate holds since the intersection vertices $\{v_i\}_{i=1}^4$ are the only vertices that can possibly contribute with positive curvature and all $2$-cells have non-positive curvature. Moreover, every vertex $v_i$ has curvature bounded above by~$\frac\pi2$. Therefore we obtain
    \[
    2\pi\le4\cdot\frac\pi2+(2-m)\pi\le\pi,
    \]
    a contradiction. It follows that $\lab v=2$. The vertex $v$ is adjacent to at least $2\cdot2=4$ vertices of type~$1$~\cite{appel1983artin}*{Lemma 6}, while a higher degree (hence a higher girth of $\link D v$) would bring an excess of negative curvature, against the combinatorial Gauss--Bonnet formula. It follows that $v$ is adjacent to precisely $4$ vertices of type~$1$. \par

    The proof in the case $v$ is a boundary vertex is analogous, the only caveat being that $\kappa(f(v))=\pi-\len{ }{\gamma}$, where $\gamma$ is a path in $\link{C_\Gamma}{f(v)}$ connecting the two vertices of $\link{C_\Gamma}{f(v)}$ corresponding to the standard tree $f(v)$ belongs to. Again, any two such distinct vertices are at distance at least $m\cdot\frac\pi2$~\cite{martin2022acylindrical}*{Lemma 4.3}.
\end{proof}

\begin{lem}
    \label{lem:cyc_std_tre_4cyc}
    Let $\Gamma$ be a labelled $4$-cycle graph. Let $\L\subseteq C_\Gamma$ be a cycle of four standard trees. The following classification holds:
    \begin{enumerate}
        \item if there are $p,q,r,s\in\N_{\ge3}$ such that $\Gamma$ is either a $(p,q,r,s)$-cycle or a $(p,q,r,2)$-cycle or a $(p,2,q,2)$-cycle, then $\L$ bounds a fundamental domain;
        \item if there are $p,q\in\N_{\ge3}$ such that $\Gamma$ is a $(p,q,2,2)$-cycle, then  $\L$ bounds either a fundamental domain or the union of two adjacent fundamental domains;
        \item if there is $p\in\N_{\ge3}$ such that $\Gamma$ is a $(p,2,2,2)$-cycle, then $\L$ bounds either one fundamental domain or the union of two adjacent fundamental domains or the square given by the union of four fundamental domains that meet at a common vertex.
    \end{enumerate}
    We refer to Figure~\ref{fig:4-cyc_std_tre_config}.
\end{lem}

\begin{figure}
    \centering
    \includegraphics[width=0.8\linewidth]{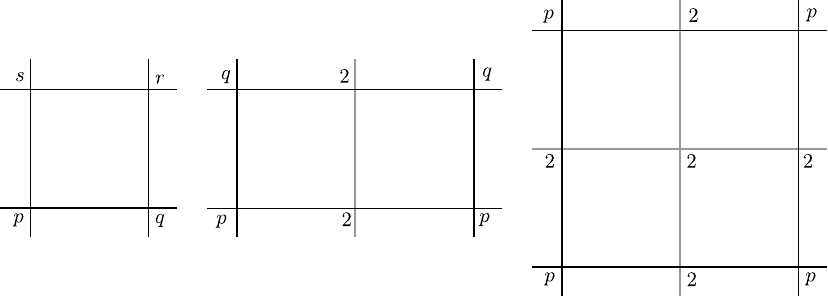}
    \caption{The possible configurations of four standard trees (in black) in the case $\girth\Gamma=4$. Each square corresponds to a fundamental domain for the action $A_\Gamma\actson C_\Gamma$.}
    \label{fig:4-cyc_std_tre_config}
\end{figure}

\begin{proof}
    The proof follows as a consequence of Lemma~\ref{lem:grid} together with the fact that $\Gamma$ is of hyperbolic type, hence it is not the $(2,2,2,2)$-square.
\end{proof}

\begin{lem}[$4$-cycles]
    \label{lem:4-cyc_art_grp}
    Let $\Gamma$ and $\Lambda$ be labelled $4$-cycles. If there is a group isomorphism $\phi\colon A_\Gamma\to A_\Lambda$, then $\Gamma\cong\Lambda$.
\end{lem}

\begin{proof}
    The solution of the isomorphism problem for RAAGs within the class of Artin groups gives that, if $\Gamma$ is the $(2,2,2,2)$-cycle, then $\Lambda$ also is~\cites{baudisch1981subgroups,droms1987isomorphisms}. Therefore, we can assume that both $\Gamma$ and $\Lambda$ have at least one label different from~$2$ and thus are of hyperbolic type. Let $\L\subseteq C_\Lambda$ be the cycle of standard trees associated with~$\phi$. \par

    If there are $p,q,r,s\in\N_{\ge3}$ such that $\Lambda$ is a $(p,q,r,s)$-cycle or a $(p,q,r,2)$-cycle or a $(p,2,q,2)$-cycle, then $\L$ bounds a translate of $K_\Lambda$, by Lemma~\ref{lem:cyc_std_tre_4cyc}. Then Lemma~\ref{lem:bound_fund_dom} implies that $\Gamma\cong\Lambda$.\par
    
    Let us now assume that there are $p,q\in\N_{\ge3}$ such that $\Lambda$ is a $(p,q,2,2)$-cycle. By what we have just proved, $\Gamma$ can only be a $(p',q',2,2)$-cycle or a $(p',2,2,2)$-cycle, for some $p',q'\in\N_{\ge3}$. By the classification of Lemma~\ref{lem:cyc_std_tre_4cyc}, the cycle $\L$ either bounds a single fundamental domain or the union of two adjacent fundamental domains. In the first case, Lemma~\ref{lem:bound_fund_dom} provides that $\Gamma\cong\Lambda$. In the second case, we would have that $A_\Gamma$ is isomorphic to the Artin group on the $4$-cycle $(p,p,q,q)$. However, being of large type is an isomorphism invariant~\cite{martin2024characterising}*{Theorem A}, a yielding contradiction. \par 

    Finally, let us assume that there is $p\in\N_{\ge3}$ such that $\Lambda$ is a $(p,2,2,2)$-cycle. Again, by what we have proved, $\Gamma$ must be a $(p',2,2,2)$-cycle, for some $p'\in\N_{\ge3}$. By the classification of Lemma~\ref{lem:cyc_std_tre_4cyc}, the cycle $\L$ either bounds a single fundamental domain, the union of two adjacent fundamental domains or the square given by the union of four fundamental domains that meet at a common vertex. In the first case, Lemma~\ref{lem:bound_fund_dom} provides that $\Gamma\cong\Lambda$. In the second and third cases we would have that $A_\Gamma$ is isomorphic to the Artin group on the $4$-cycle $(p,p,2,2)$ or $(p,p,p,p)$. Let us see that the two cases are not admissible. The groups $A_{(p,p,2,2)}$ and $A_{(p,2,2,2)}$ cannot be isomorphic, for $A_{(p,p,2,2)}$ has two distinct conjugacy classes of a $\da p$-subgroup, while $A_{(p,2,2,2)}$ has only one~\cite{godelle2007artin}*{Corollary 4.12}; the groups $A_{(p,p,p,p)}$ and $A_{(p,2,2,2)}$ cannot be isomorphic, for only one of them is of large type and being of large type is an isomorphism invariant, as argued in the previous paragraph.
\end{proof}

Finally, we deal with the triangle case.

\begin{lem}[$3$-cycles]
    \label{lem:rigid3cyc}
    Let $\Gamma$ and $\Lambda$ be labelled triangle graphs; if $A_\Gamma\cong A_\Lambda$, then $\Gamma\cong\Lambda$.
\end{lem}

\begin{proof}
    Depending on the labels of $\Lambda$, $A_\Lambda$ is either of spherical type or two-dimension\-al. If $A_\Lambda$ is of spherical type, then so is $A_\Gamma$, since being two-dimensional is an isomorphism invariant. In this case, the solution of the isomorphism problem within the class of spherical-type Artin groups implies that $\Gamma\cong \Lambda$~\cite{paris2004artin}. Let us then assume that $A_\Lambda$ (thus $A_\Gamma$) is two-dimensional.

    \begin{claim}\label{claim:large_triang}
        If $\Lambda$ is of large type (i.e. no label of $\Lambda$ is a $2$), then $\Gamma\cong\Lambda$.
    \end{claim}

    \begin{proof}[Proof of Claim~\ref{claim:large_triang}]
        If $\Lambda$ is of large type, then $\Gamma$ is also of large type~\cite{martin2024characterising}*{Theorem A}. Therefore, $\Gamma$ and $\Lambda$ are either the $(3,3,3)$-triangle or they are of hyperbolic type. If $\Lambda$ is of hyperbolic type, then any embedding $A_\Gamma\to A_\Lambda$ induces a triangle of standard trees in $D_\Lambda$. Such triangle of standard trees bounds a single translate of $K_\Lambda$~\cite{vaskou2025automorphisms}*{Lemma 3.10} and hence $\Gamma\cong\Lambda$ by Lemma~\ref{lem:bound_fund_dom}. If $\Lambda$ is the $(3,3,3)$-triangle, then the same argument applied to any embedding $A_{(3,3,3)}\to A_\Gamma$ shows that $\Gamma=(3,3,3)$.
    \end{proof}

    Henceforth we can assume that there are $p,q,u,v\in\N_{\ge3}$ such that $\Gamma=(2,p,q)$ and $\Lambda=(2,u,v)$.

    \begin{claim}
        \label{claim:hyp_triangle}
        If $\Lambda$ is of hyperbolic type, then $\Gamma\cong\Lambda$.
    \end{claim}

    \begin{proof}[Proof of Claim~\ref{claim:hyp_triangle}]
        Since $\Lambda$ is of hyperbolic type, it follows from Lemma~\ref{lem:labels_subset} that $\{p,q\}\subseteq\{u,v\}$. If $p$ and $q$ are distinct, then $\{p,q\}=\{u,v\}$ and therefore $\Gamma\cong\Lambda$. If $p=q$, then $p\ge4$, for the Artin group $A_{(2,3,3)}$ would be of spherical type. If $p\ge5$, then $\Gamma$ is of hyperbolic type and Lemma~\ref{lem:labels_subset} applied to any embedding $A_\Lambda\to A_\Gamma$ gives $\{m,n\}\subseteq\{p\}$, hence $\Gamma\cong\Lambda$. \par
        We are left to prove that, if $A_{(2,4,4)}\cong A_{(2,4,v)}$, then $v=4$. To this end, let us assume by contradiction that $v\ge 5$ (note that $A_{(2,3,4)}$ is of spherical type) and let us fix an isomorphism $\phi\colon A_{(2,4,4)}\to A_{(2,4,v)}$. Let $a,b,c$ be the standard generators of $A_{(2,4,4)}$ with $m_{bc},m_{ac}= 4$, and let $y,z$ be standard generators of $A_{(2,4,v)}$ with $m_{yz}=4$. By Lemma~\ref{lem:dihedral of hyp-type are parab}, there are $g,h\in A_\Gamma$ such that
        \[
        \text{$\phi(\Span{b,c})\le\conj{\Span{y,z}}{\phi(g)}$ and $\phi(\Span{a,c})\le\conj{\Span{y,x}}{\phi(h)}$.}
        \]
        Because $\Span{b,c}$ and $\Span{a,c}$ are not conjugate in $A_{(2,4,4)}$, it follows that $\phi^{-1}(\Span{y,z})$ is a dihedral Artin subgroup that properly contains the parabolic subgroup $\phi(g)^{-1}\Span{b,c}\phi(g)$, against Lemma~\ref{lem:dih_par_are_max}.
    \end{proof}

    We are left with the case in which both $A_\Gamma$ and $A_\Lambda$ are two-dimensional, not large and not of hyperbolic type. The only defining graphs that allow such conditions are $(2,4,4)$ and $(2,3,6)$. The groups $A_{(2,4,4)}$ and $A_{(2,3,6)}$ cannot be isomorphic, as their abelianisations do not agree.
\end{proof}

We are ready to prove Theorem~\ref{thm:rigidCycArtGrp}.

\begin{proof}[Proof of Theorem~\ref{thm:rigidCycArtGrp}]
    Let $\Gamma$ and $\Lambda$ be two labelled cycles such that $A_\Gamma\cong A_\Lambda$. By Theorem~\ref{thmintro:girIsoInv}, $\Gamma$ and $\Lambda$ are cycles of the same length~$n\in\N_{\ge3}$. Lemma~\ref{lem:rigid_girth>=5}, Lemma~\ref{lem:4-cyc_art_grp} or Lemma~\ref{lem:rigid3cyc} provide the claim, depending on whether $n\ge5$, $n=4$ or $n=3$, respectively.
\end{proof}

\bibliographystyle{abbrv}
\bibliography{bibliography}
\end{document}